\documentclass[12pt]{amsart}
\usepackage{indentfirst}
\usepackage{bm, mathrsfs, graphics,float,amssymb,amsmath,setspace,graphicx,amsthm,epstopdf,subfigure, enumerate, color, bbm}
\usepackage{thmtools, thm-restate}
\usepackage[Symbol]{upgreek}
\usepackage[hidelinks]{hyperref}
\usepackage[margin = 0.8in]{geometry}
\def\R{\mathbb{R}}

\def\P{\mathbb{P}}

\def\p{\partial}

\newcommand{\dd}{\mathrm{d}}

\newcommand{\eps}{\varepsilon}

\newcommand{\mb}{\mathbf}

\newcommand{\E}{\mathbb{E}}

\numberwithin{equation}{section}

\newtheorem{lemma}{Lemma}[section]
\newtheorem{theorem}{Theorem}[section]
\newtheorem{proposition}{Proposition}[section]
\newtheorem{definition}{Definition}[section]
\newtheorem{corollary}{Corollary}[section]
\newtheorem{remark}{Remark}[section]

\begin{document}
\title{Convergence of Oja's online principal component flow}

\author[J.-G. Liu]{Jian-Guo Liu}
\address{Department of Mathematics and Department of
  Physics, Duke University, Durham, NC}
\email{jliu@math.duke.edu}

\author[Z. Liu]{Zibu Liu}
\address{Department of Mathematics, Duke University, Durham, NC}
\email{zibu.liu@duke.edu}

\maketitle
\begin{abstract}
Online principal component analysis (PCA) has been an efficient tool in practice to reduce dimension. However, convergence properties of the corresponding ODE are still unknown, including global convergence, stable manifolds, and convergence rate. In this paper, we focus on the stochastic gradient ascent (SGA) method proposed by Oja. By regarding the corresponding ODE as a Landau-Lifshitz-Gilbert (LLG) equation on the Stiefel manifold, we proved global convergence of the ODE. Moreover, we developed a new technique to determine stable manifolds. This technique analyzes the rank of the initial datum. Using this technique, we derived the explicit expression of the stable manifolds. As a consequence, exponential convergence to stable equilibrium points was also proved. The success of this new technique should be attributed to the semi-decoupling property of the SGA method: iteration of previous components does not depend on that of later ones.
As far as we know, our result is the first complete one on the convergence of an online PCA flow, providing global convergence, explicit characterization of stable manifolds, and closed formula of exponential convergence depending on the spectrum gap.       
\end{abstract}
\section{Introduction}
The recursive method for principal component analysis (PCA) proposed by Oja in \cite{oja1992principal,oja1985stochastic} pioneered the discussion of several online PCA algorithms and their applications in the field of machine learning. Suppose that $\mb{x}\in \R^n$ is a random variable with zero expectation and distribution $\nu$. PCA aims to find a group of orthonormal vectors $\mb{w_1},\ \mb{w_2},\ ,...,\mb{w_p}\in \R^n, 1\leq p\leq n$ such that they maximize 
\begin{align} \label{eq:rayleighquotient}
	R_p:=\E\left[\sum_{i=1}^p(\mb{w_i}^T\mb{x})^2\right].
\end{align}  
Equivalently, $\mb{w_1},\ \mb{w_2},\ ...,\ \mb{w_p}$ are the $p$ dominant eigenvectors of the covariance matrix
\begin{align}
	\mb{A} := \E\left[\mb{x}\mb{x}^T\right].
\end{align} 

If $\mb{A}$ is given, then PCA reduces to diagonalizing a symmetric matrix which is a standard linear algebra problem. Nevertheless, $\mb{A}$ is generally unknown in practice. Instead, only a series of independent samplings of the random variable $\mb{x}$ are available. Moreover, due to the limitation of storage, only several or a limited amount of samplings can be addressed at the same time. Therefore, an online version of the PCA algorithm is desired: it only requires several most recent samplings of $\mb{x}$ to complete the iteration and efficiently converges to eigenvectors. The term 'online' here means that only a limited amount of samplings of $\mb{x}$ are available in one iteration. 

To solve this problem, Oja \cite{oja1992principal,oja1985stochastic} derived the following online PCA algorithm which approximates dominant eigenvectors of $\mb{A}$ efficiently. Let $\mb{W}=[\mb{w}_1,\ \mb{w_2}, ...,\ \mb{w_p}]\in\R^{n\times p}$. Given a stream of data $\{\mb{x}(k)\}$ which are independent samples of the random variable $\mb{x}$, i.e., $\mb{x}(k)\sim \nu$ are independent, consider $\mb{W}(k),k=1,2,...$ which are defined as:
\begin{gather}\label{eq:onlinePCA}
\left\{
\begin{split}
\mb{W}(0)&\in O(n\times p),\\
\tilde{\mb{W}}(k) &= \mb{W}(k-1)+\eta_k\mb{x}(k)\mb{x}^T(k)\mb{W}(k-1),\ k\geq 1,\\
\mb{W}(k) &= \tilde{\mb{W}}(k)\mb{S}(k),\ k\geq 1.
\end{split}\right.
\end{gather} 
Here $\eta_k, k=1,2,...$ are real numbers representing learning rates and $\mb{S}(k)\in\R^{p\times p}$ are matrices depending on $\tilde{\mb{W}}(k)$ which orthonormalize $\tilde{\mb{W}}(k)$. Therefore, $\mb{W}^T(k)\mb{W}(k)=\mb{I_p}$ holds for all $k$,  where $\mb{I_p}$ is the identity matrix of size $p$. 

In \eqref{eq:onlinePCA}, if $\mb{S}(k)$ conducts the Gram-Schmidt orthonormalization (GSO) on the columns of $\tilde{\mb{W}}(k)$, then the stochastic gradient ascent (SGA) algorithm is derived \cite{oja1992principal}. To reduce the computational complexity of the SGA method, authors of \cite{oja1982simplified,oja1985stochastic} developed a first order approximation of the SGA scheme:
\begin{equation} \label{eq:SGA}
	\begin{aligned} 
	\mb{w_j}(k) &= \mb{w_j}(k-1) + \eta(k)\mb{x}^T(k)\mb{w_j}(k-1)[\mb{x}(k)-(\mb{x}^T(k)\mb{w_j}(k-1))\mb{w_j}(k-1)\\
	&-2\sum_{i=1}^{j-1}(\mb{x}^T(k)\mb{w_i}(k-1))\mb{w_i}(k-1)],\ j=1,\ 2,\ ...,\ p.
	\end{aligned}
\end{equation}
Because higher-order terms are omitted, this scheme does not preserve orthogonality. This approximated scheme is easy to apply in simulations and appropriate for neural network implementation \cite{oja1992principal}. The reason is that iteration of $\mb{w_j}$ only depends on $\mb{w_i}, i=1,2,...,j-1$. We will call this property the 'semi-decoupling' property.

From now on, we assume $p=n$, i.e., we consider extracting all eigenvectors instead of only the first several dominant ones. Due to the semi-decoupling property, results for $p<n$ are actually a consequence of the one for $p=n$. This will be explained in Section \ref{sec:p<n}. 

\subsection{The corresponding differential equation for the SGA method}
Formally omitting the higher-order terms and replacing $\mb{x}(k)\mb{x}^T(k)$ by its average $\mb{A}$ in \eqref{eq:SGA} as in \cite{oja1992principal}, 
 one can derive the following differential equation of $\mb{q_i}\in \R^n,\ i=1,2,...,n$:
\begin{gather}\label{eq:ndODE}
\left\{
\begin{split}
\mb{\dot{q_1}}&=\mb{Aq_1}-\mb{(q_1\cdot Aq_1)q_1},\\
\mb{\dot{q_j}}&=\mb{Aq_j}-\mb{(q_j\cdot Aq_j)q_j}-2\sum_{i=1}^{j-1}\mb{(q_i\cdot Aq_j)q_i},\ j=2,\ 3,\ ...,\ n.\\
\mb{q_i}(0)&=\mb{q_{i,0}},\ i=1,\ 2,\ ...,\ n.
\end{split}\right.
\end{gather}
Here $\mb{Q_0}:=[\mb{q_{1,0}}, \mb{q_{2,0}},..., \mb{q_{n,0}}]\in O(n)$.  One can rewrite \eqref{eq:ndODE} in matrices by denoting $\mb Q(t)=[\mb{q_1}(t), \mb{q_2}(t), ..., \mb{q_n}(t)]$, which yields
\begin{gather} \label{eq:mtrxODE_0}
\left\{
\begin{split}
\mb{\dot{Q}} &= \mb{AQ} -\mb{Q}\sum_{j=1}^n\mb{E_j}\mb{Q}^T\mb{AQ}\mb{E_j} - 2\mb{Q}\sum_{j=1}^n\sum_{k=1}^{j-1} \mb{E_k}\mb{Q}^T\mb{AQE_j},\\
\mb{Q}(0) &= \mb{Q_0}\in O(n).
\end{split}\right.
\end{gather}
Here $\mb{E_j}, j=1,2,...,n$ represent the matrix with 1 at $j$th row, $j$th column, but 0 at other positions. 

System \eqref{eq:mtrxODE_0} admits a unique solution by standard ODE theories. In fact, one can prove that system  \eqref{eq:mtrxODE_0} defines a flow on the Stiefel manifold $O(n)$:
\begin{lemma}\label{lmm:orthogonality}
	Suppose that $\mb{Q}(t)$ is the solution to the system \eqref{eq:mtrxODE_0}. Then for any initial value $\mb{Q_0}\in O(n)$, we have $\mb{Q}(t)\in O(n)$, i.e.
	\begin{align}
		\mb{Q}(t)\mb{Q}^T(t)=\mb{Q}^T(t)\mb{Q}(t) = \mb{I_n}.
	\end{align}
	Here $\mb{I_n}\in\R^{n\times n}$ is the identity matrix of size $n$.
\end{lemma}
See proof of Lemma \ref{lmm:orthogonality} in Section \ref{sec:notations}. In this paper, we will provide a complete convergence analysis of \eqref{eq:mtrxODE_0}, including global convergence, stable manifolds, and convergence rate. In the following part of the introduction, we will first review literature and then highlight our contribution to the research of online PCA. 

\subsection{Previous results and unsolved problems}
Another PCA flow developed by Oja and Brocket \cite{brockett1991dynamical} also attracted much attention:
\begin{gather} \label{eq:Brockett}
	\left\{
\begin{split}
\mb{\dot{Q}}&=\mb{AQD}-\mb{QDQ}^T\mb{AQ},\\
\mb{Q}(0) &= \mb{Q_0}\in O(n).
\end{split}\right.	
\end{gather}
where $\mb{D}\in\R^{n\times n}$ is a symmetric matrix. As \eqref{eq:mtrxODE_0}, \eqref{eq:Brockett} is also invariant on the Stiefel manifold. Global convergence to equilibria of \eqref{eq:Brockett} was derived first in \cite{helmke2012optimization} under the assumption that $\mb{D}$ is diagonal. Later on, if $\mb{D}$ is only symmetric, global convergence was again verified in \cite{yoshizawa2001convergence}. However, to our best knowledge, similar result on the flow \eqref{eq:mtrxODE_0} is unknown. Only local convergence result \cite{oja1985stochastic} and the global convergence of the first component $\mb{p_1}$ \cite{li2017diffusion} are available. The barrier of a global analysis of \eqref{eq:mtrxODE_0} is probably due to its complicated matrix representation, unlike \eqref{eq:Brockett} which is clean. 

Some other convergence properties of both \eqref{eq:Brockett} and \eqref{eq:mtrxODE_0} also remain unknown. First, the stable manifolds of both \eqref{eq:Brockett} and \eqref{eq:mtrxODE_0} (if exists) are undetermined \cite{blondel2004unsolved}. Second, a global convergence rate was not computed for either \eqref{eq:Brockett} or \eqref{eq:mtrxODE_0}. Only local asymptotic analysis was conducted on both \eqref{eq:Brockett} \cite{helmke2012optimization,yoshizawa2001convergence} and \eqref{eq:mtrxODE_0} \cite{oja1982simplified,sanger1989optimal}. In conclusion, only a few results were derived for both \eqref{eq:Brockett} and \eqref{eq:mtrxODE_0}.

\subsection{Main results} 
First, we studied convergence property of \eqref{eq:mtrxODE_0}. In analog to the Landau-Lifshitz-Gilbert equation \cite{landau1998theory,gilbert1955lagrangian,gilbert2004phenomenological}, we recast \eqref{eq:mtrxODE_0} in the following form on a Riemannian manifold:  
\begin{align*}
\dfrac{\dd \mb{m}(t)}{\dd t} = \widehat{\dot{\mb{m}}}_1(\mb{m}(t),t) -\nabla_g\mathcal{E}(\mb{m}(t)).
\end{align*}
See \eqref{eq:LLGonM} and Proposition \ref{prop:LLG}. This equation consists of two parts: the Hamiltonian part $\widehat{\dot{\mb{m}}}_1$ and the dissipative part $-\nabla_g\mathcal{E}$, which is the variation of $-\mathcal{E}$ in some sense. They are perpendicular under the  Riemannian metric $g$. The critical relationship between them is that the Hamiltonian conserved by the Hamiltonian part is exactly the free energy $\mathcal{E}$ minimized by the dissipative part. 

To regard \eqref{eq:mtrxODE_0} as an LLG equation on the Stiefel manifold, we specified the free energy $\mathcal{E}$ and the Riemmanian metric $g$. Instead of directly using the Rayleigh quotient, we adopted a weighted version of it:
\begin{align*}
	\mathcal{E}(\mb{Q};\mb{A},\mb{N}) =\mathrm{tr}(\mb{N}\mb{Q}^T\mb{A}\mb{Q}).
\end{align*}
See \eqref{def:WRQ}. Here $\mb{N}$ is a given diagonal matrix with entries on the diagonal line aligned in a descending order. The critical advantage of utilizing $\mb N$ here is that it helps align the eigenbasis in order. In fact, by the Wielandt-Hoffman inequality (see Lemma \ref{lmm:WH}), $\mathcal{E}(\mb{Q};\mb{A},\mb{N})$ is maximized on $O(n)$ if and only if $\mb{Q}$ is exactly the eigenbasis (up to sign) of $\mb{A}$  aligned in a descending order according to the eigenvalues of $\mb A$. This is exactly the result one desires from conducting PCA. Therefore, we selected $-\mathcal{E}(\mb{Q};\mb{A},\mb{N})$ as the free energy (minimizing $-\mathcal{E}(\mb{Q};\mb{A},\mb{N})$ is equivalent to maximizing $\mathcal{E}(\mb{Q};\mb{A},\mb{N})$).

Observing that  $-\mathcal{E}(\mb{Q};\mb{A},\mb{N})$ is indeed a Lyapunov function of \eqref{eq:mtrxODE_0} (see Lemma \ref{lmm:lyapnov}), we constructed the Reimannian metric $\tilde{g}$ by the method in \cite{barta2012every}. For two tangent vector fields $\mb X$ and $\mb Y$ on the Stiefel manifold, $\tilde{g}$ is defined as 
\begin{equation*} 
\begin{aligned} 
\langle\mb{X}, \mb{Y}\rangle_{\tilde{g}}&:=\langle\mb{X_0}, \mb{Y_0}\rangle_{g_e}-\dfrac{1}{\langle\mathcal{E}',\ \mb{F}\rangle_F}\langle \mathcal{E}',\ \mb{X_1}\rangle_F\langle \mathcal{E}',\ \mb{Y_1}\rangle_F.
\end{aligned}
\end{equation*}
Here $g_e$ is the Eucledian metric, $\mb{F}$ is defined in \eqref{eq:finvODE}, $\mathcal{E}'$ is the derivative of the $\mathcal{E}$ w.r.t. $\mb Q$, $\mb{X}=\mb{X_0}+\mb{X_1}$ and $\mb{Y}=\mb{Y_0}+\mb{Y_1}$ are certain decompositions of $\mb X$ and $\mb Y$ respectively, see \eqref{eq:decomp}. For details of $\tilde{g}$, see \eqref{def:gtilde} and Proposition \ref{thm:R}.

Finally, we can reformulate \eqref{eq:mtrxODE_0} as an LLG equation on the Stiefel manifold: in \eqref{eq:LLGonM}, let $\widehat{\dot{\mb{m}}}_1=\mb 0, \mathcal{E}=-\mathcal{E}(\mb{Q};\mb{A},\mb{N}), g=\tilde{g}$. We emphasize that the Hamiltonian part $\widehat{\dot{\mb{m}}}_1$ is degenerate in \eqref{eq:mtrxODE_0}. Thus, we will explore usage of this Hamiltonian part in future.  

After the above reformulation, we are ready to study the convergence property of \eqref{eq:mtrxODE_0}. First, in Theorem \ref{thm:convergence}, we proved the global convergence of \eqref{eq:mtrxODE_0}: for arbitrary initial datum in $O(n)$, \eqref{eq:mtrxODE_0} converges to an equilibrium of the flow. 

Moreover, we comprehensively characterized the stable manifolds of \eqref{eq:mtrxODE_0}. To our best knowledge, this is the first complete result on stable manifolds for a PCA flow. To derive that, we developed a new technique based on rank analysis on the initial datum $\mb{Q}_0$. For invertible $\mb{Q}\in \R^{n\times n}$, we recursively defined $\sigma_i(\mb Q),\ i=1,2,...,n$ which is a permutation of $\{1,2,...,n\}$:
\begin{gather*} 
\sigma_1(\mb Q):=\inf\limits_{1\leq k\leq n}\{k:\ q_{1,k}\neq 0\},\\
\sigma_m(\mb Q):=\inf\limits_{1\leq k\leq n}\{k:\ \mathrm{det}(\mb{Q}[1,2,3,...,m;\  \sigma(1),\ \sigma(2),\ ...,\ k])\neq 0\}.
\end{gather*}
See \eqref{def:sigma}. The geometric meaning of $\sigma_m(\mb Q)$ is to determine which component of $\mb Q$ converges to the $m$-th eigenvector of $\mb A$. 

By Theorem \ref{thm:stablemanifold}, the convergent point of $\mb{Q}(t)$ in \eqref{eq:mtrxODE_0} with initial value $\mb{Q}(0)=\mb{Q_0}$ is given by 
\begin{align*}
	\mb{q_{\sigma_{m}(\mb{Q_0})}}(t) \to \mathrm{sgn}(z_m)\mb{e_m},\ m=1,2,...,n,
\end{align*}
here $z_m$ is a number defined in Lemma \ref{lmm:z_mrepresentation}. Geometrically, $z_m$ determines the orientation of the orthogonal frame. One can see that the convergence point of $\mb{Q}(t)$ as $t\to\infty$ only depends on $\sigma_m(\mb{Q_0}),m=1,2,...,n$, so by Corollary \ref{cor:stablerank}, we have $\sigma_m(\mb{Q}(t))=\sigma_m(\mb{Q_0}),m=1,2,...,n$ for all $t>0$, thus rank analysis is only necessary for the initial datum. 

As a corollary, we also determined the convergence rate to stable equilibria. If $\mb Q(t)$ converges to the $[\mb{e_1},\mb {e_2},...,\mb{e_n}]$ (up to sign), then by Theorem \ref{thm:expconv}, 
\begin{equation*}
	\|\mb{q_i}-\mb{e_i}\| \leq Ce^{-\nu_{i}t},\ i=1,2,...,n,
\end{equation*}
where 	
\begin{equation*}
\begin{aligned} 
\nu_k &:= \min\{\lambda_1-\lambda_2,\ \lambda_2-\lambda_3,\ ...,\ \lambda_{k}-\lambda_{k+1}\}>0,\ k=1,2,...,n-1,\\
\nu_n &:=\nu_{n-1}
\end{aligned}
\end{equation*}
and $\lambda_1>\lambda_2>...>\lambda_n$ are eigenvalues of $\mb A$. See \eqref{eq:rate} for details. 

Theorem \ref{thm:stablemanifold} and \ref{thm:expconv} reveals the mechanism of the SGA method. If $\mb Q$ converges to a stable equilibrium, because $v_k\geq v_{k+1}$, convergence of $\mb {q_k}$ is faster than that of $\mb{q_{k+1}}$ for all $k=1,2,...,n-1$. Thus, alignment of $\mb{q_1}$ to $\mb{e_1}$ is the first to complete, and due to orthogonality, $\mb{q_2},...,\mb{q_n}$ will be forced into the orthogonal complement of $\mathrm{span}\{\mb{q_1}\}$, which is very close to the subspace of $\mathrm{span}\{\mb{e_2},...,\mb{e_n}\}$. Then, alignment of $\mb{q_2}$ to $\mb{e_2}$ will be completed and  $\mb{q_3},...,\mb{q_n}$ will be forced into the orthogonal complement of $\mathrm{span}\{\mb{q_1},\mb{q_2}\}$, which is close to the subspace of $\mathrm{span}\{\mb{e_3},...,\mb{e_n}\}$. Meanwhile, in this process, alignment of $\mb{q_1}$ will not be influenced due to semi-decoupling. Alignment of other components is then conducted in order. Finally, alignment of $\mb{q_n}$ is completed and the solution converges. This is the real mechanism of the SGA method in the case where $\mb A$ has all single eigenvalues. 

Finally, results for the case $p<n$ is discussed in Section \ref{sec:p<n}. By the semi-decoupling property, all results in the case of $p=n$ can be extended to this case.

The rest of paper will be organized in the following pattern: in Section \ref{sec:notations}, we clarify assumptions and notations. In Section \ref{sec:LLG}, we introduced the Landau-Lifshitz-Gilbert equation and generalize it into Riemannian manifolds. In Section \ref{sec:PCA} and \ref{sec:ode}, we first regard \eqref{eq:mtrxODE_0} as a Landau-Lifshitz-Gilbert equation on the Stiefel manifold. Then we derive global convergence of \eqref{eq:mtrxODE_0} by using its dissipative property. Moreover, we develop other convergence properties, including the stable manifolds and the convergence rate by exploiting the semi-decoupling property. This critical property is reinterpreted as a solution formula of \eqref{eq:mtrxODE_0} which involves the Cholesky decomposition. Omitted proofs of lemmas and calculation can be found in Section \ref{sec:appendix}.

\section{Premier} \label{sec:notations}
\subsection{Reformulation of \eqref{eq:mtrxODE_0}}
Before we introduce our main results in detail, we conduct preliminary analysis of \eqref{eq:mtrxODE_0} and recast it into a more compact form. 

Notice the following identity: for all $\mb B\in\R^{n\times n}$,
\begin{align*}
	\sum_{1\leq k\leq j\leq n}\mb{E_kBE_j}+\sum_{1\leq j<k\leq n}\mb{E_kBE_j}=\sum_{1\leq j,k\leq n}\mb{E_kBE_j}=\mb{B},
\end{align*}
we have
\begin{align*}
&\ \ \ \ \sum_{j=1}^n\mb{E_j}\mb{Q}^T\mb{AQ}\mb{E_j} +2\sum_{j=1}^n\sum_{k=1}^{j-1} \mb{E_k}\mb{Q}^T\mb{AQE_j} \\
&= \sum_{j=1}^n\sum_{k=1}^{j} \mb{E_k}\mb{Q}^T\mb{AQE_j}+\sum_{j=1}^n\sum_{k=1}^{j-1} \mb{E_k}\mb{Q}^T\mb{AQE_j}\\
&= \mb{Q}^T\mb{AQ}+\sum_{j=1}^n\sum_{k=1}^{j-1} \mb{E_k}\mb{Q}^T\mb{AQE_j}-\mb{E_j}\mb{Q}^T\mb{AQE_k},
\end{align*}
thus \eqref{eq:mtrxODE_0} can be reformulated as
\begin{gather}\label{eq:mtrxODE}
\left\{
\begin{split}
\mb{\dot{Q}} &= \mb{AQ} - \mb{Q}\mb{Q}^T\mb{AQ} + \mb{Q}\sum_{j=1}^n\sum_{k=1}^{j-1}\mb{E_j}\mb{Q}^T\mb{AQE_k} - \mb{E_k}\mb{Q}^T\mb{AQE_j},\\
\mb{Q}(0) &= \mb{Q_0}:=[\mb{q_{1,0}}, \mb{q_{2,0}},..., \mb{q_{n,0}}]\in O(n).
\end{split}\right.
\end{gather}

Similarly, we can rewrite \eqref{eq:SGA} by matrices. For $\mb{\Lambda},\mb{Q}\in\R^{n\times n}$, define
\begin{equation}\label{def:G}
\begin{aligned}
\mb{\Sigma}(\mb{\Lambda},\mb{Q})&:=\sum_{j=1}^n\sum_{k=1}^{j-1}\mb{E_j}\mb{Q}^T\mb{\Lambda QE_k} - \mb{E_k}\mb{Q}^T\mb{\Lambda QE_j},\\ 
\mb{G}(\mb{\Lambda},\mb{Q})&:=\mb{\Lambda Q}-\mb{QQ}^T\mb{\Lambda}\mb{Q}+\mb{Q}\mb{\Sigma}(\mb{\Lambda},\mb{Q}).
\end{aligned}
\end{equation}
Then \eqref{eq:SGA} also reads as
\begin{gather} \label{eq:rewriteleading}
\left\{
\begin{split}
\mb{A}(k)&=\mb{x}(k)\mb{x}^T(k),\\
\mb{W}(k)&=\mb{W}(k-1)+\eta_k\mb{G}(\mb{A}(k),\mb{W}(k-1)). 
\end{split}\right.
\end{gather}

Remember that by Lemma \ref{lmm:orthogonality}, solutions to \eqref{eq:mtrxODE_0} satisfy $\mb{Q}(t)\mb{Q}^T(t)=\mb{I_n}, t>0$ if $\mb{Q}(0)\in O(n)$. We prove this lemma here. 
\begin{proof}[Proof of Lemma \ref{lmm:orthogonality}]
	Let $\mb{F}(\mb{Q})=\mb{Q}\displaystyle\sum_{j=1}^n\sum_{k=1}^{j-1}(\mb{E_j}\mb{Q}^T\mb{AQE_k} - \mb{E_k}\mb{Q}^T\mb{AQE_j})$. So $\mb{F}(\mb{Q})=\mb{Q}\mb{\Sigma({A,Q})}$. Because $\mb A$ is symmetric, so $\mb{\Sigma({A,Q})}$ is a skew-symmetric matrix. Direct computation implies that
	\begin{align}
	\dfrac{\mathrm{d}}{\mathrm{d}t}(\mb{Q}(t)\mb{Q}^T(t)) &= \mb{A}\mb{Q}\mb{Q}^T + \mb{Q}\mb{Q}^T\mb{A} - 2\mb{Q}\mb{Q}^T\mb{AQ}\mb{Q}^T + \mb{Q}\mb{\Sigma({A,Q})}\mb{Q}^T + \mb{Q}\mb{\Sigma({A,Q})}^T\mb{Q}^T.
	\end{align}
	Because $\mb{\Sigma({A,Q})}$ is skew-symmetric, thus $\mb{P}(t):=\mb{Q}(t)\mb{Q}^T(t)$ satisfies an algebraic Ricatti eqaution \cite{bittanti2012riccati}:
	\begin{align}\label{eq:RED}
	\mb{\dot{P}}(t) = \mb{AP} + \mb{PA} - 2 \mb{PAP},\ \mb{P}(0)=\mb{I_n}.
	\end{align}
	By uniqueness of solutions, we know that $\mb{P}(t)$ stays invariant, i.e. $\mb{P}(t)=\mb{I_n}$ for any $t\geq 0$. So $\mb{Q}(t)\in O(n)$.
\end{proof}
\begin{remark}
	 Notice that even though $\mb{Q}(0)$ does not stay on the Stiefel manifold, $\mb{P}(t)=\mb{Q}(t)\mb{Q}^T(t)$ also satisfies \eqref{eq:RED} with $\mb{P}(0)\neq\mb{I_n}$. Based on this observation, it is proved in \cite{yan1994global} that if rank$(\mb{Q}(0))=n$, then 
	\begin{align}\label{eq:back}
	\lim\limits_{t\to \infty}\|\mb{Q}^T(t)\mb{Q}(t)-\mb{I_n}\|_F=0.
	\end{align}
	This convergence is also exponential convergence. Main steps of this \eqref{eq:back} will be provided in Section \ref{sec:appendix}.
\end{remark}
So we can simplify \eqref{eq:mtrxODE} by cancelling the first two terms on the R.H.S.:
\begin{gather}\label{eq:smpODE}
\left\{
\begin{split}
\mb{\dot{Q}} &=  \mb{Q}\sum_{j=1}^n\sum_{k=1}^{j-1}(\mb{E_j}\mb{Q}^T\mb{AQE_k} - \mb{E_k}\mb{Q}^T\mb{AQE_j}),\\
\mb{Q}(0) &= \mb{Q_0}\in O(n).
\end{split}\right.
\end{gather}
Preserving the notation in the proof of Lemma \ref{lmm:orthogonality} and notations in \eqref{def:G}, we define
\begin{align}\label{def:F}
 \mb{F(Q)}:=\mb{Q}\mb{\Sigma}(\mb{A},\mb{Q}). 
\end{align}
Thus one can rewrite \eqref{eq:smpODE} as
\begin{gather}\label{eq:finvODE}
\left\{
\begin{split}
\mb{\dot{Q}} &= \mb{F(Q)},\\
\mb{Q}(0) &= \mb{Q_0}\in O(n).
\end{split}\right.
\end{gather}
From now on, we will use \eqref{eq:finvODE} instead of \eqref{eq:mtrxODE_0} in all proofs.
\subsection{Notations and assumptions}
In this paper, we assume that $\nu$ has compact support, i.e., there exists a constant $M>0$ such that
\begin{align} \label{ass:L_inf}
\P(\|\mb{x}\|_2\leq M)=1.
\end{align}

In the following sections, we will adopt both the matrix representation and the component-wise representation of \eqref{eq:finvODE}, thus we clarify the notation here. $\mb{q_i},\ i=1,2,...,n$ represent the column vectors of $\mb{Q}$ in order, i.e.
\begin{align}
	\mb{Q} = [\mb{q_1},\ \mb{q_2},\ ...,\  \mb{q_n}],
\end{align}  
while $\tilde{\mb{q_i}},\ i=1,2,...,n$ represent the row vectors of $\mb{Q}$ in order, i.e.
\begin{align}
	\mb{Q}^T = [\tilde{\mb{q_1}}^T,\ \tilde{\mb{q_2}}^T,\ ...,\ \tilde{\mb{q_n}}^T].
\end{align}
For each entry, $q_{i,j},\ i,j=1,2,...,n$ represent the entries at $i$th row, $j$th column of the matrix $\mb{Q}$, i.e.
\begin{align}
	\mb{q_j} = (q_{1,j},\ q_{2,j},\ ...,\  q_{n,j})^T.
\end{align} 
The canonical orthonormal basis in $\R^n$ is denoted as $\mb{e_j}, j=1,2,...,n$, which are written in column vectors, i.e.
\begin{align}
	\mb{I_n} = [\mb{e_1},\ \mb{e_2},\ ...,\ \mb{e_n}].
\end{align}
Here $\mb{I_n}$ is the identity matrix of size $n$.

For $\mb{M},\mb{N}\in\R^{n\times n}$, $\|\mb{M}\|_{F}$ represents the Frobenius norm of $\mb{M}$ and $\langle \mb{M},\ \mb{N}\rangle_{F}$ represents the inner product in the Frobenius sense:
\begin{align}
\|\mb{M}\|=\sqrt{\mathrm{tr}(\mb{M}\mb{M}^T)},\ \langle \mb{M},\ \mb{N}\rangle_{F}=\mathrm{tr}(\mb{M}^T\mb{N}).
\end{align}
For $\mb{x}=(x_1,x_2,...,x_n)\in\R^n$,  $\|x\|_2$ represents the $\ell_2$ norm of $\mb{x}$, i.e.
\begin{align}
	\|\mb{x}\|_{2}=\sqrt{\sum_{j=1}^n|x_j|^2}
\end{align}

For notation of submatrices, given $\mb{M}\in \R^{n\times n}$, row indices $\{a_1,a_2,...,a_r\}$
and column indices $\{b_1,b_2,...,b_c\}$, the submatrix that is formed from rows $\{a_1,a_2,...,a_r\}$ and columns $\{b_1,b_2,...,b_c\}$ is denoted as
\begin{align}
	\mb{M}[a_1,a_2,...,a_r;\ b_1,b_2,...,b_c].
\end{align}
The rank of a matrix $\mb{M}$ is denoted as $\mathrm{rank}(\mb{M})$.

Assume that the eigenvalues of $\mb{A}$ are all single, i.e. of multiplicity one. Denote them as
\begin{align}\label{eq:Aeigenvalue} \lambda_1>\lambda_2>...>\lambda_n>0
\end{align}
in descending order. Without loss of generality, we assume that $\mb A$ is diagonal:
\begin{gather*}
	\mb A=\mathrm{diag}\{\lambda_1,\lambda_2,...,\lambda_n\}.
\end{gather*}

By default, omitted proofs of Lemmas and other important but complicated computations are available in Section \ref{sec:appendix}.

\section{The Landau-Lifshitz-Gilbert equation} \label{sec:LLG}

\subsection{The Landau-Lifshitz-Gilbert equation in $\R^3$}
As a motivation, we first recall the Landau-Lifshitz-Gilbert equation \cite{landau1998theory,gilbert1955lagrangian,gilbert2004phenomenological} in this section. The magnetization $\mb{m}\in\R^3$ in a ferromagnet is varying at each point while preserving the magnitude, which equals to the saturation magnetization $m_s$. Landau and Lifshitz \cite{landau1998theory} proposed the following differential equation of $\mb{m}$, which describes the rotation of magnetization in reponse to external torque:
\begin{align} \label{eq:LLG}
\dfrac{\dd \mb{m}}{\dd t} = -\gamma\mb{m}\times\mb{h} - \lambda\mb{m}\times(\mb{m}\times\mb{h}).
\end{align}
Here '$\times$' represents the cross product of two vectors in $\R^3$, $\mb{h}\in\R^3$ is the effective magnetic field applied to the magnetic moment, $\gamma$ is the electron gyromagnetic ratio and $\lambda>0$ is a damping parameter which is related to $\gamma$ and $m_s$. In 1955, Gilbert \cite{gilbert1955lagrangian,gilbert2004phenomenological} modified \eqref{eq:LLG} by introducing parameters characterizing the material property. The modified version was of the same form as \eqref{eq:LLG}, but with $\gamma$ and $\lambda$ of different physical meanings. Therefore, we refer to \eqref{eq:LLG} as the Landau-Lifshitz-Gilbert (LLG) equation.

The effective magnetic field $\mb{h}$ is the negative derivative of a magnetic energy density function w.r.t the magnetization $\mb{m}$, i.e.
\begin{align}
\mb{h}=-\dfrac{\p F(\mb{m})}{\p \mb{m}}.
\end{align}
Versatile choices of the magnetic energy were considered under different physical contexts. For instance, one can choose the exchange energy $F(\mb{m})=\dfrac{1}{2}\int_{\Omega}|\nabla\mb{m}|^2\dd\mb{x}$ to arrange the molecular magnetic field in order. 

Now suppose that the energy density function $F(\mb{m})$ is given. A direct calculation yields
\begin{equation}
\begin{aligned}
\dfrac{\dd F(\mb{m}(t))}{\dd t} &= -\dfrac{\dd \mb{m}}{\dd t}\cdot \mb{h}\\
&= \gamma(\mb{m}\times\mb{h})\cdot\mb{h}+\lambda\mb{h}\cdot(\mb{m}\times(\mb{m}\times\mb{h}))\\
&= -\lambda|\mb{m}\times\mb{h}|^2\leq 0.
\end{aligned}
\end{equation}
Thus $F(\mb{m})$ is dissipated. As the first term is perpendicular to the gradient direction, i.e. $\mb{m}\times\mb{h}\perp\mb{h}$, the dissipation of the free energy is totally contributed by the second term, i.e. $-\lambda\mb{m}\times(\mb{m}\times\mb{h})$. Therefore, we refer to this term as the dissipative term. 

Another observation is that 
\begin{align}
	\dfrac{\dd\|\mb m\|^2}{\dd t}= 2\mb m\cdot (-\gamma\mb{m}\times\mb{h} - \lambda\mb{m}\times(\mb{m}\times\mb{h}))=0.
\end{align} 
Therefore, $\|\mb m\|$ is also preserved. So if $\mb m(0)\in \mathbb{S}^2$, then $\mb m(t)\in \mathbb{S}^2$ for all $t>0$.

If $\lambda=0$ in \eqref{eq:LLG}, then $F(\mb{m}(t))$ is conserved. If $F(\mb{m})$ is selected as the Kinetic energy and $\mb{m}$ represents the body angular velocity of a rigid body under free rotation, then the LLG equation reduces to the Euler equation of a rigid body. Therefore, the first term should be regarded as the Hamiltonian part of the LLG equation, which conserves the free energy $F(\mb{m})$.  

\subsection{Generalization of the LLG equation: on Riemmanian manifolds}

In the previous subsection, we observed that the critical structure of the LLG equation is that: first, it possesses both a Hamiltonian part and a dissipative part, and they are perpendicular to each other; second, the Hamiltonian for the Hamiltonian part is exactly the free energy minimized by the dissipative part, i.e. $F(\mb{m})$. Moreover, $\|\mb m\|$ is preserved, so solutions of \eqref{eq:LLG} is invariant on $\mathbb{S}^2$. 

Therefore, we desire the LLG equation on a Riemannian manifold $(\mathcal{M},g)$ should preserve the above properties:
\begin{itemize}
	\item It possesses the Hamiltonian-dissipation structure: the Hamiltonian term $H$ is perpendicular to the dissipative term $D$ on $(\mathcal{M},g)$; $D$ is the variation of a free energy $\mathcal{E}$; $H$ preserves $\mathcal{E}$.  
	\item The solution $\mb m(t)\in \mathcal{M}$. For instance, the solution of \eqref{eq:LLG} stays on $\mathbb{S}^2$ and the solution of \eqref{eq:finvODE} satys on $O(n)$.
\end{itemize}

Suppose that $(\mathcal{M},g)$ is a Riemannian manifold (see Section \ref{sec:appendix} for definition). Let $\mathcal{E}(\mb{m}):\ \mathcal{M}\to\R$ be a smooth function on $\mathcal{M}$ and $\nabla_g\mathcal{E}$ be the gradient of $\mathcal{E}$ w.r.t. the Riemannian metric $g$ (see Section \ref{sec:appendix} for definition). Then by definition of the gradient, for any tangent vector field $\dot{\mb{x}}$, we have
\begin{align}
\langle \mathcal{E}'(\mb{m}),\ \dot{\mb{x}}\rangle=\langle \nabla_g\mathcal{E},\ \dot{\mb{x}}\rangle_g.
\end{align} 
Here $\mathcal{E}'(\mb{m})$ is the derivative of $\mathcal{E}$, which is a cotangent vector field. The bracket $\langle\cdot,\cdot\rangle$ is understood as the action of a cotangent vector field (in this case, $\mathcal{E}'$) on a tangent vector field. 

The gradient flow of the free energy $\mathcal{E}$ on $(\mathcal{M},g)$ is the following differential equation:
\begin{align} \label{eq:gradientflow}
\dfrac{\dd \mb{m}(t)}{\dd t} = -\nabla_g\mathcal{E}(\mb{m}).
\end{align}
The term $-\nabla_g\mathcal{E}(\mb{m})$ will serve as the dissipative term in the LLG equation on $\mathcal{M}$. 

The Hamiltonian part is defined as following: for any time dependent tangent vector field $\dot{\mb{m}}_1(\mb{m},t):\ \mathcal{M}\times [0,\infty)\to T\mathcal{M}$, define
\begin{align} \label{def:Hamiltonian}
\widehat{\dot{\mb{m}}}_1 := \|\nabla_g\mathcal{E}(\mb{m})\|_g^2\cdot\dot{\mb{m}}_1-\langle \nabla_g\mathcal{E}(\mb{m}),\ \dot{\mb{m}}_1\rangle_g\cdot\nabla_g\mathcal{E}(\mb{m}).
\end{align} 
Then at each fixed time $t$, $\widehat{\dot{\mb{m}}}_1(\cdot,t)$ is still a tangent vector field since it is the sum of two tangent vector fields. Moreover, direct calculation yields
\begin{equation} \label{eq:0}
\begin{aligned}
\langle \nabla_g\mathcal{E}(\mb{m}),\ \widehat{\dot{\mb{m}}}_1\rangle_g &= \|\nabla_g\mathcal{E}(\mb{m})\|_g^2\langle \nabla_g\mathcal{E}(\mb{m}),\ {\dot{\mb{m}}}_1\rangle_g-\|\nabla_g\mathcal{E}(\mb{m})\|_g^2\langle \nabla_g\mathcal{E}(\mb{m}),\ {\dot{\mb{m}}}_1\rangle_g\\
&=0.
\end{aligned}
\end{equation}
Thus $\widehat{\dot{\mb{m}}}_1$ is perpendicular to the gradient direction and preserves the free energy. 

Upon \eqref{eq:gradientflow} and \eqref{def:Hamiltonian}, we are ready to develop the LLG equation on $(\mathcal{M},g)$. For any time dependent tangent vector field $\dot{\mb{m}}_1$, the following differential equation serves as an analog of \eqref{eq:LLG}, i.e. the LLG equation on the Riemannian manifold $(\mathcal{M},g)$:
\begin{align} \label{eq:LLGonM}
\dfrac{\dd \mb{m}(t)}{\dd t} = \widehat{\dot{\mb{m}}}_1(\mb{m}(t),t) -\nabla_g\mathcal{E}(\mb{m}(t)),
\end{align}
where $\widehat{\dot{\mb{m}}}_1$ is defined in \eqref{def:Hamiltonian}. Then we have the following proposition for \eqref{eq:LLGonM}:

\begin{proposition} \label{prop:LLG}
	(the LLG equation on a Riemannian manifold) Suppose that $(\mathcal{M},g)$ is a Riemannian manifold. Let $\mathcal{E}(\mb{m}):\ \mathcal{M}\to\R$ be a smooth function on $\mathcal{M}$ and $\dot{\mb{m}}_1$ be a time dependent tangent vector field $\dot{\mb{m}}_1(\mb{m},t):\ \mathcal{M}\times [0,\infty)\to T\mathcal{M}$. Consider equation \eqref{eq:LLGonM} with initial value $\mb{m}(0)=\mb{m}_0\in\mathcal{M}$, i.e.
	\begin{align*}
	\dfrac{\dd \mb{m}(t)}{\dd t} = \widehat{\dot{\mb{m}}}_1 -\nabla_g\mathcal{E}(\mb{m}),\ \mb{m}(0)=\mb{m}_0,
	\end{align*}
	where $\hat{\dot{\mb{m}}}_1$ is defined in \eqref{def:Hamiltonian}. Then: 
	\begin{enumerate}[(i)]
		\item ($\mathcal{M}$ invariance) for any $t\geq 0$, $\mb{m}(t)\in\mathcal{M}$;
		\item (energy dissipation) for any $t\geq 0$,
		\begin{align} \label{eq:energydissipation}
		\dfrac{\dd\mathcal{E}(\mb{m}(t))}{\dd t} = -\|\nabla_g\mathcal{E}(\mb{m}(t))\|_g^2\leq 0.
		\end{align}
	\end{enumerate}
\end{proposition}

\begin{proof}
	Because both $\nabla_g\mathcal{E}$ and $\dot{\mb{m}_1}$ are tangent vector fields, so $(i)$ is straightforward. For $(ii)$, from definition \eqref{def:gradient}  (the definition of gradients on Riemannian manifolds), direct calculation implies 
	\begin{equation*}
	\begin{aligned}
	\dfrac{\dd \mathcal{E}(\mb{m}(t))}{\dd t} &= \langle \mathcal{E}'(\mb{m}(t)),\ \widehat{\dot{\mb{m}}}_1 -\nabla_g\mathcal{E}(\mb{m}(t))\rangle\\
	&= \langle \nabla_g\mathcal{E}(\mb{m}(t)), \widehat{\dot{\mb{m}}}_1\rangle_g - \langle \nabla_g\mathcal{E}(\mb{m}(t)), \nabla_g\mathcal{E}(\mb{m}(t))\rangle_g.
	\end{aligned}
	\end{equation*}
	By \eqref{eq:0}, we konw that $\langle \nabla_g\mathcal{E}(\mb{m}(t)), \widehat{\dot{\mb{m}}}_1\rangle_g=0$. Therefore, 
	\begin{align*}
	\dfrac{\dd \mathcal{E}(\mb{m}(t))}{\dd t} = -\langle \nabla_g\mathcal{E}(\mb{m}(t)), \nabla_g\mathcal{E}(\mb{m}(t))\rangle_g =  -\|\nabla_g\mathcal{E}(\mb{m}(t))\|_g^2\leq 0.
	\end{align*}
	This proves \eqref{eq:energydissipation}.
\end{proof}

In summary, we generalized the LLG equation to any Riemmanian manifold. Based on it, we are going to embed the problem of online PCA into the macroscopic framework of the LLG equation on Riemannian manifolds.


\section{Online PCA: an LLG equation on the Stiefel manifold} \label{sec:PCA}

In this section, we formulate the corresponding ODE of online PCA, i.e. equation \eqref{eq:finvODE} as an LLG equation on the Stiefel manifold. Lemma \ref{lmm:orthogonality} ensures that solutions of \eqref{eq:finvODE} will stay on the Stiefel manifold if the initial value $\mb{Q}_0\in O(n)$, so we only need to select a free energy $\mathcal{E}(\mb{Q}): O(n)\to\R$ and assign an appropriate Riemannian metric $g$ to the Stiefel manifold. 

\subsection{The weighted Rayleigh quotient and the Wielandt-Hoffman inequality} 
A natural choice of the free energy is the Rayleigh quotient \cite{helmke2012optimization}:
\begin{align} \label{def:RQ}
	R_{\mb{A}}(\mb{Q}) = \mathrm{tr}(\mb{Q}^T\mb{A}\mb{Q}),
\end{align}
since maximizing the Rayleigh Quotient is equivalent to exploring the principle eigenspaces. However, if one desires to not only detect the principle eigenspaces, but also to distinguish each principle eigenvector from each other, certain modification is necessary on the Rayleigh quotient to achieve this aim. In fact, as long as $\mb{Q}\in O(n)$, the Rayleigh quotient is constant and equals to $\mathrm{tr}(\mb{A})$:
\begin{align*}
	R_{\mb{A}}(\mb{Q}) = \mathrm{tr}(\mb{Q}^T\mb{A}\mb{Q}) = \mathrm{tr}(\mb{Q}\mb{Q}^T\mb{A}) = \mathrm{tr}(\mb{A}).
\end{align*}
Therefore, the Rayleigh quotient can not be selected as the free energy directly.

To distinguish different eigenvectors, one way is to assign different weights to different components \cite{helmke2012optimization}. Let $\mb{N}\in\R^{n\times n}$ be a symmetric matrix. Consider
\begin{align} \label{def:WRQ}
	\mathcal{E}(\mb{Q};\mb{A},\mb{N}) = \mathrm{tr}(\mb{N}\mb{Q}^T\mb{A}\mb{Q}).
\end{align}
As long as $\mb{N}$ is given, this is equivalent to selecting the diagonalization of $\mb{N}$. Suppose that $\mb{P}\in O(n)$ diagonalizes $\mb{N}$, i.e.,  $\mb{N}=\mb{P}\mb{D}\mb{P}^T$ where $\mb{D}$ is diagonal. Then
\begin{equation*}
	\begin{aligned}
	\max\limits_{\mb{Q}\in O(n)}\mathrm{tr(\mb{N}\mb{Q}^T\mb{A}\mb{Q})} &= \max\limits_{\mb{Q}\in O(n)}\mathrm{tr(\mb{P}\mb{D}\mb{P}^T\mb{Q}^T\mb{A}\mb{Q})} = \max\limits_{\mb{Q}\in O(n)}\mathrm{tr(\mb{D}\mb{P}^T\mb{Q}^T\mb{A}\mb{Q}\mb{P})} \\
	&= \max\limits_{\mb{QP}\in O(n)}\mathrm{tr(\mb{D}\mb{P}^T\mb{Q}^T\mb{A}\mb{Q}\mb{P})} = \max\limits_{\mb{Q}\in O(n)}\mathrm{tr(\mb{D}\mb{Q}^T\mb{A}\mb{Q})}.
	\end{aligned}
\end{equation*} 
So without loss of generality, we assume that $\mb{N}$ is diagonal. 

To justify that maximizing \eqref{def:WRQ} does provide the eigenspace decomposition, we need the Wielandt-Hoffman inequality: 
\begin{lemma} \label{lmm:WH}
	(the Wielandt-Hoffman inequality)
	Suppose that $\mb{M},\mb{N}\in\R^{n\times n}$ are symmetric matrices. Let $\lambda_{1,i},i=1,2,...,n$ and $\lambda_{2,i},i=1,2,...,n$ be eigenvalues of $\mb{M}$ and $\mb{N}$ respectively in descending order, i.e.
	\begin{align*}
		\lambda_{1,1}&\geq\lambda_{1,2}\geq...\geq\lambda_{1,2},\\ 
		\lambda_{2,1}&\geq\lambda_{2,2}\geq...\geq\lambda_{2,n}. 
	\end{align*}
	
	Then 
	\begin{align}
		\|\mb{M}-\mb{N}\|_{F}^2\geq\sum_{i=1}^n(\lambda_{1,i}-\lambda_{2,i})^2.
	\end{align}
	The equality holds if and only if $\mb{M}$ and $\mb{N}$ can be simultaneously diagonalized with diagonal entries aligned in descending order, i.e., there exists $\mb{P}\in O(n)$ such that
	\begin{align}
		\mb{P}^T\mb{M}\mb{P}=\mathrm{diag}\{\lambda_{1,1},\lambda_{1,2},...\lambda_{1,n}\},\ \mb{P}^T\mb{N}\mb{P}=\mathrm{diag}\{\lambda_{2,1},\lambda_{2,2},...\lambda_{2,n}\}
	\end{align}
\end{lemma} 
See Section \ref{sec:appendix} for the proof of the above inequality.

By Lemma \ref{lmm:WH}, we can prove that maximizing the weighted Rayleigh quotient is exactly decomposing $\R^n$ into eigenspaces of $\mb{A}$ in order:
\begin{proposition} \label{prop:WRQ}
	Suppose that $\mb{N}=\mathrm{diag}\{\mu_1,\mu_2,...,\mu_n\}$ with $\mu_1>\mu_2>\mu_3>...>\mu_n$. Consider the weighted Rayleigh quotient defined in \eqref{def:WRQ}. Then 
	\begin{align}
		\max\limits_{\mb{Q}\in O(n)}\mathcal{E}(\mb{Q};\mb{A},\mb{N})=\sum_{i=1}^n\mu_i\lambda_i.
	\end{align}
	Here $\lambda_i$ are eigenvalues of $\mb{A}$ defined in \eqref{eq:Aeigenvalue}. The maximum is attained if and only if
	\begin{align} \label{eq:aim}
		\mb{Q}^T\mb{AQ} = \mathrm{diag}\{\lambda_1,\lambda_2,...,\lambda_n \}.
	\end{align} 
\end{proposition}
\begin{proof}
	Using the fact
	\begin{align*}
		\|\mb{Q}^T\mb{A}\mb{Q}-\mb{N}\|^2_F=\|\mb{Q}^T\mb{A}\mb{Q}\|^2+\|\mb{N}\|^2-2\mathrm{tr}(\mb{N}\mb{Q}^T\mb{A}\mb{Q}),
	\end{align*}
	we derive
	\begin{equation*}
		\begin{aligned}
		\mathcal{E}(\mb{Q};\mb{A},\mb{N}) &= \mathrm{tr}(\mb{N}\mb{Q}^T\mb{AQ})\\
		&=\dfrac{\|\mb{Q}^T\mb{AQ}\|_F^2+\|\mb{N}\|_F^2-\|\mb{Q}^T\mb{AQ}-\mb{N}\|_F^2}{2}\\
		&=\dfrac{\|\mb{A}\|_F^2+\|\mb{N}\|_F^2-\|\mb{Q}^T\mb{AQ}-\mb{N}\|_F^2}{2}.
		\end{aligned}
	\end{equation*}
	By Lemma \ref{lmm:WH}, we have 
	\begin{align}
	\|\mb{Q}^T\mb{A}\mb{Q}-\mb{N}\|_F^2 \geq \sum_{i=1}^n(\lambda_i-\mu_i)^2.
	\end{align}
	Thus
	\begin{align}
		\mathcal{E}(\mb{Q};\mb{A},\mb{N})\leq\dfrac{1}{2}\left(\sum_{i=1}^n\lambda_i^2+\mu_i^2-(\lambda_i-\mu_i)^2\right)=\sum_{i=1}^n\mu_i\lambda_i.
	\end{align}
	The equality holds if $\mb{Q}$ diagonalizes $\mb{A}$ with descending order diagonal entries, i.e. \eqref{eq:aim} holds.
\end{proof}

Lemma \ref{lmm:WH} incorporates the meaning of introducing the weight $\mb{N}$: by the rearrangement inequality, it distinguishs different eigenvectors by assigning different weights to them and forces the column of $\mb{Q}$ to align in a descending order of corresponding eigenvalues. This mechanism facilitates the eigenspace decomposition.

Indeed, $-\mathcal{E}(\mb{Q};\mb{A},\mb{N})$ is a Lyapunov function for \eqref{eq:finvODE}:  
\begin{lemma} \label{lmm:lyapnov}
	Suppose that $\mb{Q}(t),\ t\geq 0$ is the solution to \eqref{eq:finvODE} for some initial value $\mb{Q_0}\in O(n)$. Let $\mb{N}=\mathrm{diag}\{\mu_1,\mu_2,...,\mu_n\}\in\R^{n\times n}$ satisfy  
	\begin{align}
	\mu_1>\mu_2>...>\mu_n.
	\end{align}
	Define \begin{align} \label{def:E}
	E:=\{\mathrm{diag}\{\eps_1,\ \eps_2,\ ...,\ \eps_n\}\mb{P}:\ \eps_i=\pm 1,\ i=1,2,...,n,\ \mb{P}\ \mathrm{is}\ \mathrm{a}\ \mathrm{permutation}\ \mathrm{matrix.}\}.
	\end{align}
	Then
	\begin{enumerate}[(i)]
		\item (Lyapunov function) Consider the weighted Rayleigh quotient $\mathcal{E}(\mb{Q};\mb{A},\mb{N})$ defined in \eqref{def:WRQ}. Then 
		\begin{align} \label{eq:Liederivative}
		\dfrac{\dd (-\mathcal{E}(\mb{Q}(t);\mb{A},\mb{N}))}{\dd t} = -2\sum_{j=1}^n\sum_{k=1}^{j-1}(\mu_k-\mu_j)(\mb{q_k}(t)\cdot \mb{Aq_j}(t))^2\leq 0;
		\end{align}
		\item (equilibrium points) the '=' in \eqref{eq:Liederivative} holds if and only if $\mb{Q}\in E$;
		\item (stable points) $\mb{F}(\mb Q)=0$ if and only if $\mb{Q}\in E$;
		\item (finite set) $|E|=2^n\cdot n!$.
	\end{enumerate}
	
\end{lemma}
See Section \ref{sec:appendix} for the proof of this lemma.

Therefore, we will choose the weighted Rayleigh quotient as the free energy of the LLG equation.

\subsection{Construction of the Riemmanian metric}
Another important feature for the LLG equation or a gradient flow structure is the Riemmanian metric, which defines the geodesic and the distance between two points on the manifold. In this section, we will construct a Riemannian metric on the Stiefel manifold (after exclusion of a finite number of equilibria) so that \eqref{eq:finvODE} can be interpreted as a LLG equation of the weighted Rayleigh quotient (see Definition \eqref{def:WRQ}), which is indeed a Lyapunov function of \eqref{eq:finvODE}. From this Lyapunov function, we will utilize the method in \cite{barta2012every} to construct the Riemannian manifold. We will first introduce the general method in \cite{barta2012every} and then apply it to our context. 

\begin{definition} \label{def:strictLyapnov}
	Let $\mb{F}$ be a continuous tangent vector field on a manifold $M$. A continuously differentiable function $\mathcal{E}:M\to \R$ is a strict Lyapunov function for 
	\begin{align} \label{eq:generalode}
	\dot{\mb{Q}}=\mb{F(Q)},
	\end{align}
	if
	\begin{align}
	\langle\mathcal{E}'(\mb{Q}),\mb{F(Q)}\rangle<0\  \mathrm{whenever}\ \mb{Q}\in M\ \mathrm{and}\ \mb{F(Q)}\neq \mb{0}.
	\end{align}
\end{definition}

According to Theorem 1 in \cite{barta2012every}, we can now construct the Riemannian metirc $\tilde{g}$.
\begin{proposition} \label{thm:R}
	(the gradient flow on M)
	Let $M_0$ be a manifold,  $\mb{F}$ a continuous tangent vector field on $M_0$ and $\mathcal{E}:M_0\to \R$ be a continuously differentiable, strict Lyapunov function (see Definition \ref{def:strictLyapnov}) for \eqref{eq:generalode}. Then there exists a Riemannian metric $\tilde{g}$ on the open set
	\begin{align}
	M:=\{\mb{Q}\in M_0:\ \mb{F}(\mb{Q})\neq 0\}
	\end{align} 
	such that
	\begin{align}
	\nabla_{\tilde{g}}\mathcal{E} = -\mb{F}.
	\end{align}
	In particular, \eqref{eq:generalode} is a gradient system on the Riemannian manifold $(M,\tilde{g})$.
\end{proposition}
\begin{proof}
	Define the kernel of a cotangent vector field $\mathcal{F}$ at each tangent space $T_{\mb{Q}}M$ is defined as
	\begin{align}
	\mathrm{Ker}(\mathcal{F}(\mb{Q})):=\{\Omega\in T_{\mb{Q}}M:\ \langle \mathcal{F}(\mb{Q}),\Omega\rangle_F=0\}.
	\end{align}
	Equivalently, each cotangent vector field $\mathcal{F}$ is viewed as a bounded linear functional on each tangent space, and the kernel of it is the kernel of the linear functional.
	
	For every $\mb{Q}\in M$, we have $\mb{F(Q)}\neq 0$ and $\langle \mathcal{E}'(\mb{Q}),\ \mb{F(Q)}\rangle_F<0$. Thus 
	\begin{align} \label{eq:help}
	\mb{F(Q)}\notin\mathrm{Ker}\mathcal{E}'(\mb{Q}),\ \mathrm{Ker}\mathcal{E}'(\mb{Q})\neq T_{\mb Q}M. 
	\end{align}
	Because $\mathcal{E}'$ is a non-trivial cotangent vector field, so $\mathrm{dim}(\mathrm{Ran}(\mathcal{E}'))=1$, which is the dimension of $\R$ (as a linear space on $\R$), thus
	\begin{equation*}
	\begin{aligned}
	\dim(\mathrm{Ker}(\mathcal{E}'(\mb{Q}))) &=\dim(T_{\mb{Q}}M)-\dim(\mathrm{Ran}(\mathcal{E}'(\mb{Q})))\\
	&= \dim(T_{\mb{Q}}M)-1.
	\end{aligned}
	\end{equation*}
	Denote $\langle\mb{F}\rangle$ be the linear space generated by $\mb{F}$, which is a subspace of $T_{\mb{Q}}M$, thus
	\begin{align*}
	\dim(\mathrm{Ker}(\mathcal{E}'(\mb{Q})))+\dim(\langle\mb{F(\mb{Q})}\rangle)=\dim(T_{\mb{Q}}M).
	\end{align*} 
	Thus by \eqref{eq:help}, we know that the tangent bundle $TM$ is the direct sum of the bundle $\mathrm{Ker}\mathcal{E}'$ and the bundle gerenated by the vector field $\mb{F(Q)}$:
	\begin{align}
	TM=\mathrm{Ker}(\mathcal{E}')\oplus\langle\mb{F}\rangle.
	\end{align}
	For every tangent vector field $\mb{X}$ on $M$, we define the corresponding decomposition, which are well-defined and continuous:
	\begin{align} \label{eq:decomp}
	\mb{X}=\mb{X_0}+\mb{X_1},\ \mb{X_0} := \mb{X} - \dfrac{\langle \mathcal{E}',\ \mb{X}\rangle_F}{\langle\mathcal{E}',\ \mb{F}\rangle_F}\mb{F}\in\mathrm{Ker}\mathcal{E}',\ \mb{X_1}:=\dfrac{\langle \mathcal{E}',\ \mb{X}\rangle_F}{\langle\mathcal{E}',\ \mb{F}\rangle_F}\mb{F}\in\langle\mb{F}\rangle.
	\end{align} 
	Suppose that $g$ is an arbitrary Riemannian metric on \eqref{def:Rmanifold}, for every tangent vector fields $\mb{X,Y}$ on $M$, define 
	\begin{equation} \label{def:gtilde}
	\begin{aligned} 
	\langle\mb{X}, \mb{Y}\rangle_{\tilde{g}}&:=\langle\mb{X_0}, \mb{Y_0}\rangle_{g}-\dfrac{1}{\langle\mathcal{E}',\ \mb{F}\rangle_F}\langle \mathcal{E}',\ \mb{X}\rangle_F\langle \mathcal{E}',\ \mb{Y}\rangle_F\\	
	&=\langle\mb{X_0}, \mb{Y_0}\rangle_{g}-\dfrac{1}{\langle\mathcal{E}',\ \mb{F}\rangle_F}\langle \mathcal{E}',\ \mb{X_1}\rangle_F\langle \mathcal{E}',\ \mb{Y_1}\rangle_F.	
	\end{aligned}
	\end{equation}
	Because $g,\mathcal{E}'$ and $\mb{F}$ are contiuous, so is $\tilde{g}$. Moreover, for each $\mb{Q}\in M$ and any $\mb{X}\in T_{\mb{Q}}M$, we have
	\begin{align}
	\langle \mb{X},\mb{X}\rangle_{\tilde{g}} &= \langle \mb{X_0},\mb{X_0} \rangle_g - \dfrac{\langle \mathcal{E}',\ \mb{X}\rangle_F^2}{\langle\mathcal{E}',\ \mb{F}\rangle_F}\geq 0.
	\end{align}
	Here we use the fact that $\langle\mathcal{E}',\mb{F}\rangle_F\leq 0$. The quality holds if and only if $\mb{X}=0$. So $\tilde{g}$ is also positive definite. Thus $\tilde{g}$ is a Riemannian metric on $M$. Notice that $\tilde{g}$ is independent with the manifold structure of $M$, so $(M,\tilde{g})$ is a Riemannian manifold.
	
	Finally, we check the compatibility condition. For every tangent vector field on $M$, we have
	\begin{align}
	\langle \mb{F},\ \mb{X}\rangle_{\tilde g} = 0 - \dfrac{1}{\langle\mathcal{E}',\ \mb{F}\rangle_F}\langle \mathcal{E}',\ \mb{F}\rangle_F\langle \mathcal{E}',\ \mb{X}\rangle_F = -\langle\mathcal{E}',\ \mb{X}\rangle_F.
	\end{align}
	Thus $\nabla_{\tilde g}\mathcal{E}=-\mb{F}$.
\end{proof}
\begin{remark}
	In fact, the Riemannian metric $\tilde{g}$ is not defined at points where $\langle\mathcal{E}'(\mb Q),\mb{F}(\mb Q)\rangle_F=0$. This discontinuous point can be removable or irremovable depending on selection of the original metric $g$, see \cite{barta2012every}. Meanwhile, the Riemannian metric is not degenerate around equilibrium points due to $\langle\mathcal{E}',\ \mb{F}\rangle_F$ on the denominator. 
\end{remark}

Now we focus on the context of \eqref{eq:finvODE} and the Stiefel manifold. Remember that we select $-\mathcal{E}(\mb{Q};\mb{A},\mb{N})$ in \eqref{def:WRQ} as the free energy. Now consider
\begin{gather} \label{def:Rmanifold}
M := O(n)\setminus E.
\end{gather} 
Here $E$ is defined in \eqref{def:E}. 

By Lemma \ref{lmm:lyapnov}, we have: first, because $|E|<\infty$, $M$ is an open subset of the Stiefel manifold. Second, we know that for all $\mb{Q}\in M$, we have $\mb{F}(\mb Q)\neq 0$ and $\langle \mathcal{E}'(\mb{Q};\mb{A},\mb{N}), \mb{F(Q)}\rangle< 0$. Thus $\mathcal{E}$ is a strict Lyapunov function on $M$. These properties allow us to apply Proposition \ref{thm:R} on $M$ and $\mathcal{E}(\mb{Q};\mb{A},\mb{N})$ in \eqref{def:WRQ}:
\begin{theorem} \label{thm:RofStiefel}
	Consider $\mb{F}$ in \eqref{def:F}, $M$ in \eqref{def:Rmanifold} and $\mathcal{E}(\mb{Q};\mb{A},\mb{N})$ in \eqref{def:WRQ}. Then there exists a Riemannian metric $\tilde{g}$ on $M$ such that
	\begin{align} \label{eq:gradient}
	\nabla_{\tilde{g}}\mathcal{E} = -\mb{F}.
	\end{align}
\end{theorem}
\begin{proof}
	Let $g$ be the Euclidean metric, i.e., for each $\mb{Q}\in M$ and $\mb{X,Y}\in T_{\mb{Q}}M$, let
	\begin{align*}
	\langle \mb{X},\mb{Y}\rangle_{g} = \langle \mb{X},\mb{Y}\rangle_{F}. 
	\end{align*} 

	Notice that $\mathcal{E}$ in \eqref{def:WRQ} is a smooth strict Lyapunov function of \eqref{eq:finvODE}, and $\mb{F}$ defines a tangent vector field on $M$, so by Proposition \eqref{thm:R}, $\tilde{g}$ defined in \eqref{def:gtilde} is a Riemannian metric on $M$ and \eqref{eq:gradient} holds.
\end{proof}
\begin{remark}
	The choice of $g$ in the proof above is arbitrary as long as it is a Riemmanian metric on the Stiefel manifold $O(n)$.
\end{remark}
\subsection{The LLG equation} Now after selecting the free energy and assigning the Riemannian metric, we can formulate the ODE of the SGA method, i.e. \eqref{eq:finvODE} as a LLG equation on the Stiefel manifold with a Riemmanian metric.

Suppose $\mb{Q}\in O(n)$. Then the tangent space  at $\mb{Q}$, i.e. $T_{\mb{Q}}O(n)$ is 
\begin{gather}
	T_{\mb{Q}}O(n) =\{\mb{Q}\bm{\Omega}:\ \bm{\Omega}+\ \bm{\Omega}^T=0\}. 
\end{gather}
Then any tangent vector field on $O(n)$ can be written in the form $\mb{QS(Q)}$, where $\mb{S(Q)}\in\R^{n\times n}$ satisfies $\mb{S(Q)}+\mb{S(Q)}^T=0$. Then by Proposition \ref{prop:LLG}, we know that 
\begin{gather} \label{eq:LLGOn}
\left\{
\begin{split}
\dfrac{\dd \mb{Q}(t)}{\dd t} &= \widehat{\mb{Q}\mb{S(Q)}} + \nabla_{\tilde{g}}\mathcal{E}(\mb{Q};\mb{A},\mb{N}),\\
\widehat{\mb{Q}\mb{S(Q)}} &= \|\nabla_{\tilde{g}}\mathcal{E}(\mb{Q};\mb{A},\mb{N})\|^2_{\tilde{g}}\cdot \mb{Q}\mb{S(Q)} - \langle \nabla_{\tilde{g}}\mathcal{E}(\mb{Q};\mb{A},\mb{N}), \mb{Q}\mb{S(Q)}\rangle_{\tilde{g}}\cdot\nabla_{\tilde{g}}\mathcal{E}(\mb{Q};\mb{A},\mb{N}),\\
\mb{Q}(0)&=\mb{Q_0}\in O(n)
\end{split}\right.
\end{gather}
is an LLG equation on the Riemannian manifold $(O(n)\setminus E,\tilde{g})$. Here the hat operator is defined in \eqref{def:Hamiltonian}, $E$ is the stable points of $\mb{F}(\mb{Q})$ defined in \eqref{def:E} and  $\tilde{g}$ is the Riemannian metric defined in \eqref{def:gtilde} in Proposition \ref{thm:R}, by taking $g$ as the Euclidean metric. If we substitute the explicit expression of $\tilde{g}$ in \eqref{def:gtilde}, \eqref{eq:LLGOn} then reads as
\begin{gather} \label{eq:LLGtildeg}
	\left\{
	\begin{split}
	\dfrac{\dd \mb{Q}(t)}{\dd t} &= \widehat{\mb{Q}\mb{S(Q)}} + \mb{F(Q)},\\
	\widehat{\mb{Q}\mb{S(Q)}} &=2\mathrm{tr}(\mb{NQ}^T\mb{AF(Q)})\mb{QS(Q)}-2\mathrm{tr}(\mb{NQ}^T\mb{AQS(Q)})\mb{F(Q)},\\
	\mb{Q}(0)&=\mb{Q_0}\in O(n).
	\end{split}\right.
\end{gather}
Again, $\mb{S(Q)}\in\R^{n\times n}$ is arbitrary smooth (w.r.t. $\mb{Q}$) skew-symmetric matrices field. The convergence behavior of this system will be thoroughly studied in the following sections.

Equation \eqref{eq:finvODE} is a special case of \eqref{eq:LLGOn} by taking $\mb{S}=0$, which reduces to the gradient flow system. Although the Riemmanian metric $\tilde{g}$ is constructed on $O(n)\setminus E$ (see Definition \eqref{def:E}), the LLG equation is well defined on the whole Stiefel manifold: if $\mb{Q_0}\in E$, then the flow stays invariant, i.e. $\mb{Q}(t)=\mb{Q_0}$.

In \cite{helmke2012optimization}, the Oja-Brockett flow was also formulated as the gradient flow of the weighted Rayleigh quotient on $(O(n),g_e)$, where $g_e$ is the Euclidean metric. The gradient of $\mathcal{E}(\mb{Q};\mb{A},\mb{N})$ on $(O(n),g_e)$ is given by
\begin{equation}
	\begin{aligned}
	\nabla_{g_e}\mathcal{E}(\mb{Q};\mb{A},\mb{N})&=\dfrac{1}{2}(\mathcal{E}'(\mb{Q};\mb{A},\mb{N})-\mb{Q}(\mathcal{E}'(\mb{Q};\mb{A},\mb{N}))^T\mb{Q})\\
	&=\mb{AQN}-\mb{QNQ}^T\mb{AQ}.
	\end{aligned}
\end{equation}
See Section \ref{sec:appendix} for calculation of gradients in $(O(n),g_e)$. If we replace $\tilde{g}$ by the Euclidean metric in \eqref{eq:LLGOn}, then we have
\begin{gather} \label{eq:LLGg}
	\left\{
	\begin{split}
	\dfrac{\dd \mb{Q}(t)}{\dd t} &= \widehat{\mb{Q}\mb{S(Q)}} + \mb{AQN-QN}\mb{Q}^T\mb{AQ},\\
	\widehat{\mb{Q}\mb{S(Q)}} &=2\mathrm{tr}(\mb{Q}^T\mb{AQN}^2-(\mb{Q}^T\mb{AQN})^2)\mb{Q}\mb{S(Q)}\\
	&-\mathrm{tr}(\mb{NQ}^T\mb{AQS(Q)}-\mb{Q}^T\mb{AQNS(Q)})(\mb{AQN-QN}\mb{Q}^T\mb{AQ}),\\
	\mb{Q}(0)&=\mb{Q_0}\in O(n).
	\end{split}\right.
\end{gather}
If $\mb{S}=\mb{0}$, then we recover the Oja-Brockett flow \cite{brockett1991dynamical}. 
\section{Convergence analysis of the LLG equation on the Stiefel manifold} \label{sec:ode}

\subsection{Convergence to equilibria}
Notice that the number of stable points of \eqref{eq:finvODE} is finite, we can prove the following convergence theorem:
\begin{theorem} \label{thm:convergence}
	(convergence) For any initial value $\mb{Q_0}\in O(n)$, suppose that $\mb{Q}(t),\ t\geq 0$ is the solution of \eqref{eq:LLGOn} or \eqref{eq:LLGg}. Then there exists $\mb{Q}^*\in E$ such that
	\begin{align} \label{eq:convergenceODE}
		\lim\limits_{t\to\infty}\mb{Q}(t)=\mb{Q}^*.
	\end{align}
	Here $E$ is defined in \eqref{def:Rmanifold}.
\end{theorem}
\begin{proof}
	If $\mb{Q_0}\in E$, then $\mb{Q}(t)=\mb{Q_0}$ for any $t\geq 0$, hence $\mb{Q}^*=\mb{Q_0}$. If $\mb{Q_0}\notin E$, by Proposition \ref{prop:LLG} and Lemma \eqref{lmm:lyapnov}, we have
	\begin{align*}
		-\dfrac{\dd\mathcal{E}(\mb{Q}(t);\mb{A},\mb{N})}{\dd t} = -\|\nabla_{g}\mathcal{E}(\mb{Q}(t);\mb{A},\mb{N})\|_{g}^2\leq 0,
	\end{align*}
	where $g=\tilde{g}$ for \eqref{eq:LLGOn} and $g=g_e$ for \eqref{eq:LLGg}. 
	
	By Lasalle's invariance principle \cite{lasalle1960some}, the $\omega$-limit set of \eqref{eq:LLGOn} or \eqref{eq:LLGg} is a subset of zeros of $\nabla_{g}\mathcal{E}(\mb{Q}(t);\mb{A},\mb{N})$. For \eqref{eq:LLGOn}, $\nabla_g\mathcal{E}(\mb{Q}(t);\mb{A},\mb{N})=\mb{F(Q)}$, so by Lemma \ref{lmm:lyapnov}, zeros of the gradient is exactly the set $E$; for \eqref{eq:LLGg}, by \eqref{eq:gradient}, any point in the $\omega-$limit set should solve 
	\begin{align*}
		\mb{AQN}-\mb{QN}\mb{Q}^T\mb{AQ}=0,
	\end{align*}
	or equivalently $\mb{Q^T}\mb{AQN}=\mb{N}\mb{Q}^T\mb{AQ}$. Because $\mb{N}$ is a diagonal matrix with different diagonal entries, so $\mb{Q}^T\mb{AQ}$ is also diagonal, thus $\mb{Q}\in E$. Therefore, for both \eqref{eq:LLGOn} and \eqref{eq:LLGg}, the $\omega-$limit set should be a subset of $E$.
	
	Because $O(n)$ is compact, so the $\omega-$limit set should be connected. Thus $\mb{Q}(t)$ converges to a connected component of $E$. However, by Lemma \ref{lmm:lyapnov}, $E$ consists of only finite points (in fact $n!2^n$ of them), so there exists $\mb{Q}^*\in E$ such that \eqref{eq:convergenceODE} holds.
\end{proof}

\subsection{Stable manifolds} 
From now on, we focus on \eqref{eq:finvODE} instead of the general case \eqref{eq:LLGOn} and proceed to discuss the stable manifolds of each equilibrium. Even though a complete characterization of the stable manifolds for the Oja-Brockett flow \cite{brockett1991dynamical} ($\mb{S=0}$ in \eqref{eq:LLGg}) is still an open problem \cite{blondel2004unsolved}, the stable manifolds of \eqref{eq:finvODE} will be fully determined in this section. The critical technique which facilitates our characterization of the stable manifolds is the following solution formula.

\subsubsection{The solution formula by the Cholesky decomposition}
First of all, we rewrite \eqref{eq:finvODE} as
\begin{gather}\label{eq:rewriteODE}
\left\{
\begin{split}
\mb{\dot{Q}} &= \mb{AQ}-\mb{QT(Q)},\\
\mb{Q}(0) &= \mb{Q_0}\in O(n),
\end{split}\right.
\end{gather}
where $\mb{T}$ is defined as
\begin{align} \label{def:T}
	\mb{T(Q)} &:= \mb{Q}^T\mb{AQ} + \sum_{k=1}^n\sum_{j=1}^{k-1}\mb{E_j}\mb{Q}^T\mb{AQE_k} - \mb{E_k}\mb{Q}^T\mb{AQE_j}\\
	&= \mathrm{diag}\{\mb{q_1}\cdot \mb{Aq_1},\ \mb{q_2}\cdot \mb{Aq_2},\ ...,\ \mb{q_n}\cdot \mb{Aq_n}\} +2\sum_{k=1}^n\sum_{j=1}^{k-1}\mb{E_j}\mb{Q}^T\mb{AQE_k}.
\end{align}
So $\mb{T}$ is an upper triangular matrix. This property determines that the evolution of $\mb{q_i}$ only depends on $\mb{q_j},\ 1\leq j<i$ but is independent with $\mb{q_k},\ i<k\leq n$. This independence is helpful to the description of the stable manifolds. 

We first derive a representation formula of $\mb{Q}$ from \eqref{eq:rewriteODE}. 
\begin{lemma} \label{lmm:representation}
	Consider \eqref{eq:finvODE} (or the alternative form \eqref{eq:rewriteODE}). Then there exists an upper triangular matrix $\mb{G}(t)$ such that
	\begin{align} \label{eq:representation}
		\mb{Q}(t) = e^{\mb{A}t}\mb{Q_0}\mb{G}(t).
	\end{align}
	Moreover, diagonal entries of $\mb{G}(t)$ are positive and $\mb{G}(t)$ satisfies
	\begin{align} \label{eq:G}
	\mb{G}(t)\mb{G}^T(t) = \mb{Q_0}^Te^{-2\mb{A}t}\mb{Q_0}.
	\end{align} 
\end{lemma}

See Section \ref{sec:appendix} for the proof of this lemma.

Notice that \eqref{eq:Giteration} should just be understood as an iteration formula coupled with $\mb{Q}$ instead of an explicit closed formula of the solution. However, $\mb{G}(t)$ is uniquely determined by \eqref{eq:G}, since \eqref{eq:G} implies that $(\mb{G}^T)^{-1}(t)$ is the Cholesky decomposition of $\mb{Q_0}^Te^{2\mb At}\mb{Q_0}$, i.e.
\begin{align}
	(\mb{G}^T)^{-1}(t)\mb{G}^{-1}(t)=\mb{Q_0}^Te^{2\mb At}\mb{Q_0}.
\end{align}

\subsubsection{Characterization of stable manifolds}

We first introduce the definition of stable manifolds for readers' convenience.

\begin{definition} \label{def:stablem}
 	Let $M$ be a smooth manifold and $a\in M$ is an equilibrium of a smooth vector vector field $X: M\to TM$. The stable manifold of $a$ (or the inset of $a$) is 
 	\begin{align}
 	W^s(a) := \{x\in M:\ L_\omega(x)=\{a\}\}. 
 	\end{align}
 	Here $L_{\omega}(x)$ is the $\omega$ limit set of $x$.
 \end{definition}

We will exploit \eqref{eq:representation} in this section to determine the stable manifolds of each equilibria of $\mathcal{E}$. For $\mb{Q}\in \R^{n\times n}$ which is invertible, we recursively define $\sigma(i),\ i=1,2,...,n$ which is a permutation of $\{1,2,...,n\}$:
\begin{gather} \label{def:sigma}
	\sigma_1(\mb Q):=\inf\limits_{1\leq k\leq n}\{k:\ q_{1,k}\neq 0\},\\
	\sigma_m(\mb Q):=\inf\limits_{1\leq k\leq n}\{k:\ \mathrm{det}(\mb{Q}[1,2,3,...,m;\  \sigma(1),\ \sigma(2),\ ...,\ k])\neq 0\}.
\end{gather}
For instance, if
\begin{gather}
	\mb{M}=
	\begin{pmatrix}
	0 & 1 & 0\\
	1 & 0 & 1\\
	0 & 1 & 1
	\end{pmatrix},
\end{gather}
then $\sigma_1(\mb M)=2,\ \sigma_2(\mb{M})=1,\ \sigma_3(\mb M)=3$.

A complete characterization of the stable manifolds of \eqref{eq:finvODE} will be derived upon a dedicate analysis of properties of $\sigma$. We state the following lemmas for $\sigma$ whose proofs are attached to the Appendix (see Section \ref{subsec:proofs}). In the following sections, by default, if the proof is not presented after a lemma then it is summarized in Section \ref{subsec:proofs}.
\begin{lemma} \label{lmm:rank}
	Let $\mb{Q}\in\R^{n\times n}$ be invertible. Then for any $1\leq m,j\leq n$, we have
	\begin{align} \label{eq:rank}
		\mathrm{rank}\left(\mb{Q}[1,2,3,...,m;\ 1,2,...,j]\right)+\mathrm{card}(\{k:\ \sigma_k(\mb{Q})>j,\ 1\leq k\leq m\})=m.
	\end{align}
\end{lemma}
See Section \ref{sec:appendix} for proof of this lemma. 
From Lemma \ref{lmm:rank}, we derive the following lemma which will be directly used to determine the stable manifolds:
\begin{lemma} \label{lmm:z_mrepresentation} 
	Let $\mb{Q}\in\R^{n\times n}$ be invertible. Define permutation $\sigma_m,\ m=1,2,...n$ as in \eqref{def:sigma}, and
	\begin{align} \label{def:i_m}
	\mb{i_m}:=(i_{m,1},\ i_{m,2},\ ...,\ i_{m,m}), 1\leq m\leq n
	\end{align}
	where $\mb{i_m}$ is a permutation of $\{\sigma_k(\mb{Q_0})\}_{1\leq k\leq m}$ such that $i_{m,1}<i_{m,2}<i_{m,3}<...<i_{m,m}$. Then
	\begin{gather} \label{eq:q_m}
	\mb{Q}[m;\ 1,2,3,...,\sigma_m(\mb{Q})]=
	\left\{
	\begin{split}
	&\mb{c_m}\mb{Q}[1,2,...,m-1;\ 1,2,3,...,\sigma_m({\mb{Q}})] + (0,0,0,...,0,z_m),\ 2\leq m\leq n,\\
	&(0,0,...,0,z_m),\ m=1.
	\end{split}
	\right.
	\end{gather}
	where $\mb{c_m}=(c_{m,1},\ c_{m,2},\ ...,\ c_{m,m-1}), 2\leq m\leq n$ is the unique solution of
	\begin{align}
	\mb{Q}[m;\ \mb{i_{m-1}}]=\mb{c_m}\mb{Q}[1,2,...,m-1;\ \mb{i_{m-1}}],
	\end{align}
	and $z_m$ is defined as	
	\begin{gather} \label{def:z_m}
	z_m=
	\left\{
	\begin{split}
	&\dfrac{\det(\mb{Q}[1,2,3,...,m;\ \mb{i_m}])}{\det(\mb{Q}[1,2,3,...,m-1;\ \mb{i_{m-1}}])},\ 2\leq m\leq n,\\
	&q_{1,\sigma_1(\mb{Q})},\ m=1.
	\end{split}\right.
	\end{gather}
\end{lemma}

See Section \ref{sec:appendix} for proof of this lemma. By Lemma \ref{lmm:z_mrepresentation}, we have the following theorem on stable manifolds:
\begin{theorem} \label{thm:stablemanifold}
	Suppose that $\mb{Q}(t), t\geq 0$ solves \eqref{eq:finvODE} with initial value $\mb{Q_0}\in O(n)$. Define permutation $\sigma_m,\ m=1,2,...n$ as in \eqref{def:sigma}, and
	\begin{align}
		\mb{i_m}:=(i_{m,1},\ i_{m,2},\ ...,\ i_{m,m}), 1\leq m\leq n
	\end{align}
	where $\mb{i_m}$ is a permutation of $\{\sigma_k(\mb{Q_0})\}_{1\leq k\leq m}$ such that $i_1<i_2<i_3<...<i_m$.
	Then 
	\begin{align} \label{eq:stablemanifolds}
		\mb{q_{\sigma_{m}(\mb{Q_0})}}(t) \to \mathrm{sgn}(z_m)\mb{e_m},\ m=1,2,...,n,
	\end{align}
	as $t\to\infty$. Here $z_m$ is defined in $\eqref{def:z_m}$
\end{theorem}
\begin{proof}
	We prove by induction. For $m=1$, we know that $q_{1,j}(0)=0$ for all $1\leq j<\sigma_1(\mb{Q_0})$ while $q_{1,\sigma_1(\mb{Q_0})}\neq 0$. By Lemma \ref{lmm:representation}, we have
	\begin{align}
		q_{1,\sigma_1(\mb{Q_0})}(t) =\dfrac{ e^{\lambda_1t}q_{1,\sigma_1(\mb{Q_0})}(0)}{f_{\sigma_1(\mb{Q_0})}(t)}.
	\end{align}
	Here $f$ is defined in \eqref{def:f}.
	By Theorem \ref{thm:convergence}, we know that $q_{1,\sigma_1(\mb{Q_0})}(t)$ converges to one of $1,\ -1$ and $0$. Because $\mb{q_{\sigma_1(\mb{Q_0})}}\cdot \mb{Aq_{\sigma_1(\mb{Q_0})}}\leq \lambda_1$, so 
	\begin{align}
		\dfrac{ e^{\lambda_1t}}{f_{\sigma_1(\mb{Q_0})}(t)} = e^{\int_0^t\lambda_1-\mb{q_{\sigma_1(\mb{Q_0})}}(s)\cdot \mb{Aq_{\sigma_1(\mb{Q_0})}}(s)\dd s}\geq e^{0}=1.
	\end{align}
	Thus $q_{1,\sigma_1(\mb{Q_0})}(t)$ does not converge to 0, and it converges to either $1$ or $-1$ which depends on the sign of the initial value $q_{1,\sigma_1(\mb{Q_0})}$. This implies
	\begin{align}
		\mb{q_{\sigma_{1}(\mb{Q_0})}}\to \mathrm{sgn}(q_{1,\sigma_1(\mb{Q_0})})\mb{e_1}
	\end{align}
	as $t\to\infty$. Thus the claim holds for $m=1$.
	
	Now suppose that \eqref{eq:stablemanifolds} holds for all $j=1,2,...,m-1$, we prove it for $j=m$. By \eqref{eq:q_m} in Lemma \ref{lmm:z_mrepresentation} and  \eqref{eq:reprformula} in Lemma \ref{lmm:representation}, we have
	\begin{align} \label{eq:q_mconv}
	\nonumber
		q_{m,\sigma_m(\mb{Q_0})}(t) &= e^{\lambda_m t}\sum_{l=1}^{\sigma_m(\mb{Q_0})}q_{m,l}(0)g_{l,\sigma_m(\mb{Q_0})}(t)\\
	\nonumber
		&= e^{\lambda_m t}\mb{Q_0}[m;\ 1,2,3,...,\sigma_m(\mb{Q_0})]\mb{G}(t)[1,2,...,\sigma_m(\mb{Q_0});\ \sigma_m(\mb{Q_0})]\\
	\nonumber
		&= e^{\lambda_m t}\mb{c_m}\mb{Q_0}[1,2,...,m-1;\ 1,2,3,...,\sigma_m({\mb{Q_0}})]\mb{G}(t)[1,2,...,\sigma_m(\mb{Q_0});\ \sigma_m(\mb{Q_0})]\\
	\nonumber
		&+\dfrac{e^{\lambda_m t}z_m}{f_{\sigma_m(\mb{Q_0})}(t)}\\
		&=\sum_{j=1}^{m-1}e^{(\lambda_m-\lambda_j)t}c_{m,j}q_{j,\sigma_m(\mb{Q_0})}(t)+\dfrac{e^{\lambda_m t}z_m}{f_{\sigma_m(\mb{Q_0})}(t)}.
	\end{align}
	Here $\mb{c_m}$ and $z_m$ are defined in Lemma \ref{lmm:z_mrepresentation}. By the hypothesis of induction, we know that
	$\mb{q_{\sigma_m(\mb{Q_0})}}$ does not converge to $\pm\mb{e_1},\ \pm\mb{e_1},\ ,...\pm\mb{e_{m-1}}$, so
	\begin{align}
	q_{j,\sigma_m(\mb{Q_0})}(t)\to 0\ \mathrm{as}\ t\to\infty.
	\end{align}
	Thus 
	\begin{align}
		\sum_{j=1}^{m-1}e^{(\lambda_m-\lambda_j)t}c_{m,j}q_{j,\sigma_m(\mb{Q_0})}(t)\to 0\ \mathrm{as}\ t\to\infty.
	\end{align}
	Again, by Theorem \ref{thm:convergence}, $q_{m,\sigma_m(\mb{Q_0})}$ converges to one of $1, -1$ and 0. Now we prove that it converges to either $1$ or $-1$. Otherwise $\mb{q_{\sigma_m(\mb{Q_0})}}$ can only converge to $\pm\mb{e_p}, p \geq m+1$, hence $\mb{q_{\sigma_m(\mb{Q_0})}}\cdot \mb{A\mb{q_{\sigma_m(\mb{Q_0})}}}\to \lambda_{p}<\lambda_m$, which results in 
	\begin{align*}
		\dfrac{e^{\lambda_m t}}{f_{\sigma_m(\mb{Q_0})}(t)}=e^{\int_0^t \lambda_m-\mb{q_{\sigma_m(\mb{Q_0})}}(s)\cdot \mb{A\mb{q_{\sigma_m(\mb{Q_0})}}}(s)\dd t}\to\infty\ \mathrm{as}\ t\to\infty,
	\end{align*}
	a contradiction. Thus $\dfrac{e^{\lambda_m t}z_m}{f_{\sigma_m(\mb{Q_0})}(t)}$ converges to either $1$ or $-1$ and 
	\begin{align}
		q_{m,\sigma_m(\mb{Q_0})}(t)=\dfrac{e^{\lambda_m t}z_m}{f_{\sigma_m(\mb{Q_0})}(t)}\to \mathrm{sgn}(z_m).
	\end{align}
	Thus \eqref{eq:stablemanifolds} holds for $m$. By induction, it holds for all $m=1,2,...,n$.
\end{proof}
As a direct corollary, we can see that $\sigma_m(\mb Q(t)), m=1,2,...,n$ does not change for all $t>0$.
\begin{corollary}\label{cor:stablerank}
	Suppose that $\mb{Q}(t)$ is a solution of \eqref{eq:finvODE}. Then for all $m=1,2,...,n$, we have
	\begin{align*}
		\sigma_m(\mb{Q}(t))=\sigma_m(\mb{Q}(0)), t>0.
	\end{align*}
\end{corollary}
\begin{proof}
	For arbitrary $t_0>0$, consider  $\mb{R}(t):=\mb{Q}(t+t_0), t>0$. Then $\mb{R}(t),t>0$ is the unique solution to \eqref{eq:finvODE} with initial value $\mb{Q}(t_0)$. Then by Theorem \ref{thm:stablemanifold}, for all $m=1,2,...,n$, we have
	\begin{align*}
		\mb{r}_{\sigma_m(\mb{R}(0))}(t)\to\mathrm{sgn}(z_m')\mb{e_m}\ \mathrm{as}\ t\to\infty.
	\end{align*}
	Remember that 
	\begin{align*}
		\mb{q}_{\sigma_m(\mb{Q}(0))}(t)\to\mathrm{sgn}(z_m)\mb{e_m}\ \mathrm{as}\ t\to\infty.
	\end{align*}
	Moreover, $\mb{R}(t)=\mb{Q}(t+t_0)$, so 
	\begin{align*}
		\lim\limits_{t\to\infty}\mb{R}(t)=\lim\limits_{t\to\infty}\mb{Q}(t).
	\end{align*}
	Thus 
	\begin{align*}
		\lim\limits_{t\to\infty}\mb{r}_{\sigma_m(\mb{R}(0))}(t)=	\lim\limits_{t\to\infty}\mb{q}_{\sigma_m(\mb{Q}(0))}(t).
	\end{align*}
	So
	\begin{align*}
		\sigma_m(\mb{R}(0))=\sigma_m(\mb{Q}(0)).
	\end{align*}
	By orthogonality of $\mb{Q}$ and $\mb{R}(t)=\mb{Q}(t+t_0)$, we know 
	\begin{align*}
		\sigma_m(\mb{Q}(t_0))=\sigma_m(\mb{R}(0))=\sigma_m(\mb{Q}(0)).
	\end{align*} 
\end{proof}
\subsubsection{An example} To illustrate Theorem \ref{thm:stablemanifold}, we discuss an example here. Consider 
\begin{align*}
	\mb{Q_1}=
	\begin{pmatrix}
	0 & \dfrac{\sqrt{2}}{2} & -\dfrac{\sqrt{3}}{3} & \dfrac{\sqrt{6}}{6}\\
	0 & \dfrac{\sqrt{2}}{2} & \dfrac{\sqrt{3}}{3} & -\dfrac{\sqrt{6}}{6}\\
	-\dfrac{\sqrt{2}}{2} & 0 & \dfrac{\sqrt{6}}{6} & \dfrac{\sqrt{3}}{3}\\
	\dfrac{\sqrt{2}}{2} & 0 & \dfrac{\sqrt{6}}{6} & \dfrac{\sqrt{3}}{3}
	\end{pmatrix}.
\end{align*}
Suppose that $\mb{Q}(t), t\geq 0$ solves \eqref{eq:finvODE} with initial value $\mb{Q_1}$. According \eqref{def:sigma}, we have
\begin{align*}
\sigma_1(\mb{Q_1})=2,\ \sigma_2(\mb{Q_1})=3,\ \sigma_3(\mb{Q_1})=1,\ \sigma_4(\mb{Q_1})=4.
\end{align*}
by Theorem \ref{thm:stablemanifold}, we have
\begin{align*}
	\mb{q_1}(t)&\to \mathrm{sgn}(z_3)\mb{e_3},\ \mb{q_2}(t)\to\mathrm{sgn}(z_1)\mb{e_1},\\
	\mb{q_3}(t)&\to \mathrm{sgn}(z_2)\mb{e_2},\ \mb{q_4}(t)\to\mathrm{sgn}(z_4)\mb{e_4}.
\end{align*}
A direction computation according to \eqref{def:z_m} yields $z_1=\sqrt{2}/2, z_2=2\sqrt{3}/3, z_3=-\sqrt{2}/2, z_4=2\sqrt{3}$. Thus as $t\to\infty$, 
\begin{align*}
\mb{q_1}(t)\to -\mb{e_3},\ \mb{q_2}(t)\to\mb{e_1},\ \mb{q_3}(t)\to \mb{e_2},\ \mb{q_4}(t)\to\mb{e_4}.
\end{align*}
\subsection{Convergence rate to the stable equilibria}
Applying the technique of linearization, we can prove that an equilibrium $\mb{Q}^*\in E$ (defined in \eqref{def:Rmanifold}) is stable if and only if
\begin{align}\label{eq:stable}
	\mb{Q}^*=\mathrm{diag}\{\eps_1,\eps_2,\eps_3,...,\eps_n\},\ \eps_i = \pm 1,\ i=1,2,...,n.
\end{align}
See Section \ref{subsec:linearization} for details. By Theorem \ref{thm:stablemanifold}, the solution of \eqref{eq:finvODE} with initial value $\mb{Q_0}$ converges to a stable equilibrium if and only if 
\begin{align}
	\sigma_k(\mb{Q_0})=k,\ k=1,2,...,n,
\end{align}
or equivalently, every leading principle submatrix is invertible, i.e.
\begin{align} \label{eq:stablecondition}
	\mathrm{det}(\mb{Q_0}[1,2,...,k;\ 1,2,...,k])\neq 0,\ k=1,2,...,n.
\end{align}

Now we proceed to prove the exponential convergence to a stable equilibrium. The main idea of this proof is similar to the one of power method: utilizing $(\lambda_i/\lambda_1)^n\to 0$ (here $e^{(\lambda_i-\lambda_1)t}\to 0$), we derive convergence rate. 
 
\begin{theorem} \label{thm:expconv}
	(exponential convergence) Suppose that $\mb{Q_0}$ is in the stable manifold of a stable equilibrium, i.e., \eqref{eq:stablecondition} holds. Let $\mb{Q}(t), t\geq 0$ be the solution of \eqref{eq:finvODE} with initial data $\mb{Q_0}$. Define
	\begin{equation}\label{eq:rate}
		\begin{aligned} 
		\nu_k &:= \min\{\lambda_1-\lambda_2,\ \lambda_2-\lambda_3,\ ...,\ \lambda_{k}-\lambda_{k+1}\}>0,\ k=1,2,...,n-1,\\
		\nu_n &:=\nu_{n-1}.
		\end{aligned}
	\end{equation}
	Then there exists a constant $C>0$ such that
	\begin{align} \label{eq:expconv}
		|q_{i,j}^2(t)-\delta_{i,j}| \leq Ce^{-2\nu_{i\wedge j}t},\ i,j=1,2,...,n.
	\end{align}
	Here $\delta_{i,j},\ i,j=1,2,...,n$ are Kronecker functions.
\end{theorem}
\begin{proof}
	In the following proof, $C$ just represents for a general constant which needs not to be invariant through the whole proof. 
	
	We will prove by induction. As we pointed out in introduction, alignment of $\mb{q_1}$ is the fastest, so we consider it first.	
	
	We first prove \eqref{eq:expconv} for $q_{k,1},q_{1,k}\ k=1,2,...,n$. By the computation in Lemma \ref{lmm:representation}, we have 
	\begin{align*}
	\mb{q_1}(t)=\dfrac{e^{\mb{A}t}\mb{q_{1,0}}}{\|e^{\mb{A}t}\mb{q_{1,0}}\|}.
	\end{align*} 
	Remember that $\mb{Q_0}$ satisfies \eqref{eq:stablecondition}, so $q_{1,1}(0)\neq 0$ and
	\begin{align*}
		\dfrac{q_{k,1}(t)}{q_{1,1}(t)}=\dfrac{q_{k,1}(0)}{q_{1,1}(0)}e^{(\lambda_k-\lambda_1)t},\ k=1,2,...,n.
	\end{align*}
	So $|q_{k,1}(t)|\leq Ce^{(\lambda_k-\lambda_1)t}|q_{1,1}(t)|$. Summing up $k$ from $2$ to $n$ yields
	\begin{align*}
		1-|q_{1,1}(t)|^2 & = \sum_{k=2}^nq_{k,1}^2(t) = C\sum_{k=2}^ne^{2(\lambda_k-\lambda_1)t}|q_{1,1}(t)|^2 \leq Ce^{2(\lambda_2-\lambda_1)t}|q_{1,1}(t)|^2.
	\end{align*}
	Thus
	\begin{align*}
		1-|q_{1,1}^2(t)|\leq \dfrac{Ce^{2(\lambda_2-\lambda_1)t}}{1+Ce^{2(\lambda_2-\lambda_1)t}} \leq Ce^{2(\lambda_2-\lambda_1)t}.
	\end{align*}
	Thus by orthogonality, we have
	\begin{equation*}
		\begin{aligned}
		|q_{1,k}(t)|^2\leq 1-|q_{1,1}(t)|^2 \leq Ce^{2(\lambda_2-\lambda_1)t},\ k=2,...,n,\\
		|q_{k,1}(t)|^2\leq 1-|q_{1,1}(t)|^2 \leq Ce^{2(\lambda_2-\lambda_1)t},\ k=2,...,n.
		\end{aligned}
	\end{equation*}
	Thus \eqref{eq:expconv} holds for $i,j$ such that $i\wedge j=1$.
	
	After alignment of first $k-1$ components, we consider that of $\mb{q_k}$.
	Suppose that \eqref{eq:expconv} holds for $i,j$ such that $i\wedge j=1,2,...,k-1, k\geq 2$, we prove that \eqref{eq:expconv} also holds for $i,j$ such that $i\wedge j=k$. By \eqref{eq:reprformula}, we have for $m=1,2,...,n$
	\begin{align*}
		e^{-\lambda_mt}q_{m,k}(t) = \sum_{j=1}^ng_{j,k}(t)q_{m,j}(0)=\mb{Q_0}[m;\ 1,2,...,k]\mb{G}(t)[1,2,...,k;\ m].
	\end{align*}
	Here $\mb{G}(t)$ is defined in Lemma \ref{lmm:representation} (see \eqref{eq:representation}). Remember that $\mb{Q_0}[1,2,...,k;\ 1,2,...,k]$ is invertible, so there exist constants $c_{m,j}, j=1,2,...,k$ such that for all $k+1\leq m\leq n$, 
	\begin{align*}
		\mb{Q_0}[m;\ 1,2,...,k]=\sum_{j=1}^kc_{m,j}\mb{Q_0}[j;\ 1,2,...,k].
	\end{align*}
	Thus for $k+1\leq m \leq n$, 
	\begin{align*}
		e^{-\lambda_mt}q_{m,k}(t) &= \sum_{j=1}^k c_{m,j}\mb{Q_0}[j;\ 1,2,...,k]\mb{G}(t)[1,2,...,k;\ m]\\
		&= \sum_{j=1}^k c_{m,j}e^{-\lambda_j t}q_{j,k}(t),
	\end{align*}
	so by induction hypothesis (i.e., \eqref{eq:expconv} holds for $i\wedge j=k-1$)
	\begin{align*}
		|q_{m,k}(t)|\leq Ce^{(\lambda_m-\lambda_k)t}|q_{k,k}(t)|+Ce^{-\nu_{k-1}t}.
	\end{align*}
	Summing up from $m=k+1$ to $n$ yields
	\begin{align*}
		1-\sum_{j=1}^kq_{j,k}^2(t)-q_{k,k}^2(t) = \sum_{m=k+1}^n |q_{m,k}(t)|^2 \leq Ce^{2(\lambda_{k+1}-\lambda_k)t}|q_{k,k}(t)|^2+Ce^{-2\nu_{k-1}t}.
	\end{align*}
	Thus
	\begin{align*}
		1-q_{k,k}^2(t) \leq  \sum_{j=1}^kq_{j,k}^2(t)+Ce^{2(\lambda_{k+1}-\lambda_k)t}q_{k,k}^2(t)+Ce^{-2\nu_{k-1}t}\leq Ce^{2(\lambda_{k+1}-\lambda_k)t}q^2_{k,k}(t)+Ce^{-2\nu_{k-1}t},
	\end{align*}
	hence
	\begin{align*}
		1-q_{k,k}^2(t)\leq \dfrac{Ce^{2\nu_{k-1}t}+Ce^{2(\lambda_{k+1}-\lambda_k)t}}{1+e^{2(\lambda_{k+1}-\lambda_k)t}}\leq Ce^{-2\min\{\nu_{k-1},\lambda_{k}-\lambda_{k+1}\}t}=Ce^{-2\nu_kt}.
	\end{align*}
	So for any $j\geq k+1$, by orthogonality,
	\begin{align*}
		q_{j,k}^2(t)\leq 1-q^2_{k,k}(t)\leq Ce^{-2\nu_kt},\\
		q_{k,j}^2(t)\leq 1-q^2_{k,k}(t)\leq Ce^{-2\nu_kt}.
	\end{align*}
	Thus \eqref{eq:expconv} also holds for $i,j$ such that $i\wedge j=k+1$. So by induction, \eqref{eq:expconv} holds for $i,j=1,2,...,n$.
\end{proof}
\subsection{The case of $p<n$}\label{sec:p<n}
Now we discuss the case $p<n$. In this case, the ODE system is reduced to 
\begin{gather}\label{eq:pdODE}
\left\{
\begin{split}
\mb{\dot{q_1}}&=\mb{Aq_1}-\mb{(q_1\cdot Aq_1)q_1},\\
\mb{\dot{q_j}}&=\mb{Aq_j}-\mb{(q_j\cdot Aq_j)q_j}-2\sum_{i=1}^{j-1}\mb{(q_i\cdot Aq_j)q_i},\ j=2,\ 3,\ ...,\ p.\\
\mb{q_i}(0)&=\mb{q_{i,0}},\ i=1,\ 2,\ ...,\ p.
\end{split}\right.
\end{gather}
The initial value $\mb{Q_0}\in\R^{n\times p}$ satisfies $\mb{Q_0}^T\mb{Q_0}=\mb{I_p}$. Denote the solution of \eqref{eq:pdODE} with initial value $\mb{Q_0}$ as $\mb{Q}(t),t>0$. 

To derive the convergence property of $\mb{Q}(t)$, we just need to complete $\mb{Q_0}$, make it in $O(n)$. Now let $\mb{R_0}\in\R^{n\times(n-p)}$ satisfy that
\begin{align}
	\mb{Q_1}=[\mb{Q_0},\mb{R_0}]\in O(n).
\end{align}

Let $\mb{Q_1}(t)$ be the solution of \eqref{eq:finvODE} with initial value $\mb{Q_1}$. Notice that \eqref{eq:pdODE} is exactly the same as the first $p$ components of \eqref{eq:ndODE} and \eqref{eq:finvODE}, thus it is also semi-decoupling. Therefore, by uniqueness the solution,
\begin{align*}
	\mb{Q}(t)=\mb{Q_1}[1,2,...,n;1,2,...,p](t),\ t>0,
\end{align*}
i.e., the first $p$ columns of $\mb{Q_1}$ and $\mb{Q}$ are the same. So the same convergence properties in Theorem \ref{thm:stablemanifold} and Theorem \ref{thm:expconv} hold for $\mb{Q}$.

\section*{Acknowledgement}
Jian-Guo Liu was supported in part by the National Science Foundation (NSF) under award DMS-2106988.

\bibliographystyle{plain}
\bibliography{PCAODE}

\section{Appendix} \label{sec:appendix}
\subsection{Riemannian manifolds and the Stiefel manifold} \label{subsec:stiefel}
We denote the tangent space at $\mb{m}$ on manifold $\mathcal{M}$ as $T_{\mb{m}}\mathcal{M}$, the tangent vector field on $\mathcal{M}$ as $\Gamma(T\mathcal{M})$. The tangent bundle (the disjoint union of the tangent spaces) is denoted as $T\mathcal{M}$.
\begin{definition} (Riemannian manifolds)
	Suppose that $\mathcal{M}$ is a smooth manifold. A Riemannian manifold $(\mathcal{M},g)$ is a smooth mainfold equipped with an inner product $g_{\mb{m}}$ on $T_{\mb{m}}\mathcal{M}$ at each $\mb{m}\in\mathcal{M}$. Moreover, for any tangent vector field $\dot{\mb{x}}$ and $\dot{\mb{y}}$, the function 
	\begin{align}
	\langle\dot{\mb{x}}(\mb{m}),\dot{\mb{y}}(\mb{m})\rangle_{g_{\mb{m}}}:\ \mathcal{M}\to\R
	\end{align}
	is smooth.
\end{definition}

Given a Riemannian metric $g$ on $M$, the gradient of a smooth function $\mathcal{E}$ on $\mathcal{M}$ is defined as 
\begin{definition} \label{def:gradient}
	(the gradient on the Riemannian manifold) A tangent vector field $\nabla_g\mathcal{E}$ on $M$ is called the gradient of $\mathcal{E}$ w.r.t. the metric $g$ if for every tangent vector field $\dot{\mb{x}}$ on $M$,
	\begin{align} \label{eq:grad}
	\langle \mathcal{E}',\ \dot{\mb{x}}\rangle_F = \langle \nabla_g\mathcal{E},\ \dot{\mb{x}}\rangle_g.
	\end{align} 
	Here $\mathcal{E}'$ is the derivative of $\mathcal{E}$, which is a cotangent vector field.
\end{definition}

Now we consider the Stiefel manifold with the Euclidean metric $g_e$, under the global coordinate $\mb{Q}\in\R^{n\times n}$. We first introduce several important properties of $O(n)$. See \cite{helmke2012optimization} for more details.

\begin{lemma} \label{lmm:Stiefel}
	(the Stiefel manifold)
	The Stiefel manifold $O(n)$ is a smooth, compact manifold of dimension $n(n-1)/2$. The tangent space at $\mb{Q}$ is given by
	\begin{align}\label{eq:tangent}
		T_{\mb{Q}}O(n)=\{\mb{Q}\bm{\Omega}\ |\ \bm{\Omega}\in\R^{n\times n},\ \bm{\Omega}+\bm{\Omega}^T=\mb{0}\},
	\end{align}
while the normal space at $\mb{Q}$ is given by 
\begin{align}\label{eq:normal}
T_{\mb{Q}}O(n)^{\perp}=\{\mb{Q}\bm{\Omega}\ |\ \bm{\Omega}\in\R^{n\times n},\ \bm{\Omega}=\bm{\Omega}^T\}.
\end{align}
\end{lemma}

See \cite{helmke2012optimization} for proof of Lemma \ref{lmm:Stiefel}. By Lemma \ref{lmm:Stiefel}, we can prove that for any $\mb{M}\in\R^{n\times n}$, the projection on the tangent spaces and the normal spaces are respectively:
\begin{align}
	\mathcal{P}_{T_{\mb{Q}}O(n)}\mb{M}:=\dfrac{1}{2}(\mb{M}-\mb{Q}\mb{M}^T\mb{Q}),\ \mathcal{P}_{T_{\mb{Q}}O(n)^{\perp}}\mb{M}:=\dfrac{1}{2}(\mb{M}+\mb{Q}\mb{M}^T\mb{Q}).
\end{align}

\begin{lemma} (Gradient on $(O(n),g_e)$) Suppose that $\varphi(\mb{Q}):O(n)\to\R$ is a restriction of a smooth function (still denoted as $\varphi:\R^{n\times n}\to\R$) on $O(n)$. Then gradient of $\varphi$ w.r.t. $g_e$ at point $\mb{Q}$ is given by
	\begin{align}
		\nabla_{g_e}\varphi:=\mathcal{P}_{T_{\mb{Q}}O(n)}(\nabla\varphi)=\dfrac{1}{2}(\nabla\varphi-\mb{Q}(\nabla\varphi)^T\mb{Q}),
	\end{align}  
	here $
	\nabla\varphi\in\R^{n\times n}$ is the gradient of $\varphi$ in $\R^{n\times n}$, i.e. $(\nabla\varphi)_{ij}=\dfrac{\p\varphi}{\p q_{ij}}$.
\end{lemma}
\begin{proof}
	We just need to prove that for any $\mb{Q}\in O(n)$ and any tangent vector at $\mb{Q}$, \eqref{eq:grad} holds. By Lemma \ref{lmm:Stiefel}, a tangent vector at $\mb{Q}$ can be represented as $\mb{Q}\bm{\Omega}$ where $\bm{\Omega}$ is skew-symmetric. Therefore, for any $\bm{\Omega}$ that is skew-symmetric, we have 
	\begin{align}
		\langle \nabla\varphi, \mb{Q}\bm{\Omega}\rangle_F = \langle \mathcal{P}_{T_{\mb{Q}}O(n)}(\nabla\varphi),\ \mb{Q}\bm{\Omega}\rangle_F + \langle \mathcal{P}_{T_{\mb{Q}}O(n)^{\perp}}(\nabla\varphi),\ \mb{Q}\bm{\Omega}\rangle_F.
	\end{align} 
	Notice that 
	\begin{align*}
		\langle \mathcal{P}_{T_{\mb{Q}}O(n)^{\perp}}(\nabla\varphi),\ \mb{Q}\bm{\Omega}\rangle_F &= \langle \mb{Q}(\mb{Q}^T\nabla\varphi+(\nabla\varphi)^T\mb{Q}),\ \mb{Q}\bm{\Omega}\rangle_F\\
		&= \langle\mb{Q}^T\nabla\varphi+(\nabla\varphi)^T\mb{Q},\ \bm{\Omega}\rangle_F\\
		&=0,
	\end{align*}
	since $\langle \mb{M},\ \mb{N}\rangle_F=0$ if $\mb{M}$ is symmetric while $\mb{N}$ is skew-symmetric. Therefore,
	\begin{align*}
		\langle \nabla\varphi, \mb{Q}\bm{\Omega}\rangle_F = \langle \mathcal{P}_{T_{\mb{Q}}O(n)}(\nabla\varphi),\ \mb{Q}\bm{\Omega}\rangle_F = \langle \mathcal{P}_{T_{\mb{Q}}O(n)}(\nabla\varphi),\ \mb{Q}\bm{\Omega}\rangle_{g_e}.
	\end{align*}
	So $\nabla_{g_e}\varphi=\mathcal{P}_{T_{\mb{Q}}O(n)}(\nabla\varphi)$.
\end{proof}
Under this global coordinate, the divergence on $(O(n), g_e)$ can also be explicitly computed. The $ij$ entry of $\nabla_{g_e}\varphi$ is given by
\begin{align}
	(\nabla_{g_e}\varphi)_{ij}= \dfrac{1}{2}\left(\dfrac{\p\varphi}{\p q_{ij}}-\sum_{k,l}q_{ik}q_{lj}\dfrac{\p\varphi}{\p q_{lk}}\right).
\end{align}
So given a tangent vector field $\mb{H}(\mb{Q})=(h_{ij}(\mb{Q}))$, the divergence of it is defined by
\begin{align}
	\nabla_{g_e}\cdot\mb{H}(\mb{Q}):=\sum_{i,j}(\nabla_{g_e}(h_{ij}(\mb{Q})))_{ij}.
\end{align}
The Laplace-Beltrami operator is then defined as:
\begin{align}
	\Delta_{g_e}\varphi:=\nabla_{g_e}\cdot\nabla_{g_e}\varphi.
\end{align}
Explicit expression of the Laplace-Beltrami operator is given by the following lemma:
\begin{lemma} \label{lmm:LBO}
	(Laplace-Beltrami operator) The Laplace-Beltrami operator on $(O(n),g_e)$ is given by
	\begin{align} \label{eq:LBO}
		\Delta_{g_e}\varphi=\dfrac{1}{2}\left(\sum_{i,j}\dfrac{\p^2\varphi}{\p{q_{ij}^2}}-(n-1)\sum_{i,j}q_{ij}\dfrac{\p\varphi}{\p q_{ij}}-\sum_{i,j,k,l}q_{ik}q_{lj}\dfrac{\p^2\varphi}{\p q_{ij}\p q_{lk}}\right).
	\end{align}
\end{lemma}
\begin{proof}
	By the expression of the divergence and gradient, we have 
	\begin{align*}
	\Delta_{g_e}\varphi=\sum_{i,j}\left(\dfrac{\p h_{ij}}{\p q_{ij}}-\sum_{k,l}q_{ik}q_{lj}\dfrac{\p h_{ij}}{\p q_{lk}}\right),\ h_{ij}=\sum_{i,j}\left(\dfrac{\p\varphi}{\p q_{ij}}-\sum_{k,l}q_{ik}q_{lj}\dfrac{\p\varphi}{\p q_{lk}}\right).
	\end{align*}
	Thus 
	\begin{align*}
		\dfrac{\p h_{ij}}{\p q_{ij}} &= \dfrac{1}{2}\left(\dfrac{\p^2\varphi}{\p q_{ij}^2}-\sum_{k,l}\delta_{kl}q_{lj}\dfrac{\p\varphi}{\p q_{lk}}-\sum_{k,l}\delta_{il}q_{ik}\dfrac{\p\varphi}{\p q_{lk}}-\sum_{k,l}q_{ik}q_{lj}\dfrac{\p^2\varphi}{\p q_{ij}q_{lk}}\right)\\
		&= \dfrac{1}{2}\left(\dfrac{\p^2\varphi}{\p q_{ij}^2}-\sum_{l}q_{lj}\dfrac{\p\varphi}{\p q_{lj}}-\sum_{k}q_{ik}\dfrac{\p\varphi}{\p q_{ik}}-\sum_{k,l}q_{ik}q_{lj}\dfrac{\p^2\varphi}{\p q_{ij}q_{lk}}\right),
	\end{align*}
	and 
	\begin{align*}
		\dfrac{\p h_{ij}}{\p q_{lk}}=\dfrac{1}{2}\left(\dfrac{\p^2\varphi}{\p q_{ij}\p q_{lk}}-\sum_{k',l'}\delta_{il}\delta_{kk'}q_{l'j}\dfrac{\p\varphi}{\p q_{l'k'}}-\sum_{k',l'}\delta_{ll'}\delta_{kj}q_{ik'}\dfrac{\p\varphi}{\p q_{l'k'}}-\sum_{k',l'}q_{ik'}q_{l'j}\dfrac{\p^2\varphi}{\p q_{l'k'}\p q_{lk}}\right).
	\end{align*}
	Substituting the above formulas into the Laplacian operator, we can derive \eqref{eq:LBO}.
\end{proof}

\subsection{Proofs of lemmas and omitted calculations} \label{subsec:proofs}
\subsubsection{Section \ref{sec:notations}}
We will provide main steps of proof of \eqref{eq:back}. Details can be found in \cite{yan1994global}. Remember that $\mb{Q}(t)$ is now not necessarily be orthogonal, but only invertible. 

First, by Lemma 2.1 in \cite{yan1994global}, an explicit solution formula of \eqref{eq:RED} is given by 
\begin{align}\label{eq:solution}
	\mb{P}(t)=e^{\mb At}\mb P(0)[\mb{I_n}+(e^{2\mb{A}t}-\mb{I_n})\mb{P}(0)]^{-1}e^{\mb At}.
\end{align}
Moreover, notice that $\mb{Q}^T(0)(e^{2\mb At}-\mb{I_n})\mb{Q}(0)$ is positive definite for all $t>0$, so it only adimts non-negative eigenvalues. Remember that $\mb{AB}$ and $\mb{BA}$ have same non-zero eigenvalues, so eigenvalues of $(e^{2\mb{A}t}-\mb{I_n})\mb{P}(0)=(e^{2\mb{A}t}-\mb{I_n})\mb{Q}(0)\mb{Q}^T(0)$ are non-negative. Thus $\mb{I_n}+(e^{2\mb{A}t}-\mb{I_n})\mb{P}(0)$ is invertible, and solution of \eqref{eq:mtrxODE_0} exists for all time. Moreover, \eqref{eq:solution} implies that rank$(\mb{P}(t))=n$ for all $t>0$. 

Second, we prove that under the assumption that $\mb{Q}(0)$ is invertible, the singular values of $\mb{Q}(t)$ are all positive and bounded from below uniformly for all $t>0$. Let orthogonal $\mb{U}$ satisfy
\begin{align*}
	\mb{Q}(0)\mb{Q}^T(0)=\mb{U}\mb{L}\mb{U}^T,
\end{align*}
here $\mb{L}$ is diagonal. Let $\mb{M}=\mb{U}e^{2\mb At}\mb{U}^T$, then by \eqref{eq:solution}, direct calculation yields
\begin{align*}
	\mb{P}(t)=e^{\mb At}\mb{U}^T(\mb{L}^{-1}+\mb{M}-\mb{I_n})^{-1}\mb{U}e^{\mb At}.
\end{align*} 
Because $\mb{Q}(0)$ is invertible and $e^{\mb At}\succ\mb{I_n}$, thus both $\mb{M-I_n}$ and $\mb {L}^{-1}$ are positive definite, so are $(\mb{L}^{-1}+\mb{M}-\mb{I_n})^{-1}$ and $\mb{P}(t)$. So $\mb{P}(t)$ is positive definite. Following the steps in Theorem 2.2 of \cite{yan1994global}, we can prove that singular values of $\mb{Q}(t)$ lower bounded uniformly for $t>0$. Let $\alpha>0$ be the lower bound.

Finally, let $v(t)=\|\mb{I_n}-\mb{Q}^T(t)\mb{Q}(t)\|_F^2$, direct calculation yields
\begin{align}
\dot{v}(t)=-4\mathrm{tr}((\mb{I_n}-\mb{Q}^T(t)\mb{Q}(t))\mb{Q}^T(t)\mb{A}\mb{Q}(t)(\mb{I_n}-\mb{Q}^T(t)\mb{Q}(t)))\leq -4\alpha^2\lambda_nv(t).
\end{align}  
By Gronwall's inequality, we proved exponential convergence under the Frobenius norm.

\subsubsection{Section \ref{sec:PCA}}
\begin{proof} [Proof of Lemma \ref{lmm:WH}]
	Suppose that $\mb{P_1}$ and $\mb{P_2}$ are orthogonal matrices which diagonalizes $\mb{M}$ and $\mb{N}$ respectively:
	\begin{align*}
	\mb{M}=\mb{P_1}^T\bm{\Lambda_1}\mb{P_1},\ \mb{N}=\mb{P_2}^T\bm{\Lambda_2}\mb{P_2},
	\end{align*} 
	here $\bm{\Lambda_1}$ and $\bm{\Lambda_2}$ are diagonal matrices and the entries on the diagonal line in descending order:
	\begin{align*}
	\bm{\Lambda_1}=\mathrm{diag}\{\lambda_{1,1},\lambda_{1,2},...\lambda_{1,n}\},\ \lambda_{1,1}\geq \lambda_{1,2}\geq ...\geq\lambda_{1,n},\\ \bm{\Lambda_2}=\mathrm{diag}\{\lambda_{2,1},\lambda_{2,2},...\lambda_{2,n}\},\ \lambda_{2,1}\geq \lambda_{2,2}\geq ...\geq\lambda_{2,n}.
	\end{align*}
	Notice that $\|\mb{PM}\|_F=\|\mb{MP}\|_F=\|\mb{M}\|$ holds for any orthogonal matrix $\mb{P}$, therefore 
	\begin{align*}
	\|\mb{M}-\mb{N}\|_F^2 &= \|\mb{P_1}^T\bm{\Lambda_1}\mb{P_1}-\mb{P_2}^T\bm{\Lambda_2}\mb{P_2}\|_F^2\\
	&= \|\mb{P_1}^T(\bm{\Lambda_1}\mb{P_1P_2}^T-\mb{P_1P_2}^T\bm{\Lambda_2})\mb{P_2}\|_F^2\\
	&= 	\|\bm{\Lambda_1}\mb{P_1P_2}^T-\mb{P_1P_2}^T\bm{\Lambda_2}\|_F^2.
	\end{align*}
	Let $\mb{Q}=\mb{P_1P_2}^T$, then $\mb{Q}$ is also an orthogonal matrix and 
	\begin{align}
	\|\mb{M}-\mb{N}\|_F^2 = \sum_{1\leq i,j\leq n}|q_{i,j}|^2(\lambda_{1,i}-\lambda_{2,j})^2.
	\end{align}
	Let $\mb{W}\in\R^{n\times n}$ be $w_{i,j}=|q_{i,j}|^2$ for $i,j=1,2,...,n$. Then $\mb{W}$ is a doubly stochastic matrix since $\mb{Q}$ is orthogonal. Here a matrix $\mb{W}$ is doubly stochastic if and only if both $\mb{W}$ and $\mb{W}^T$ are transition matrices. 
	
	Denote the set of all doubly matrices in $\R^{n\times n}$ as $\mathcal{D}_n$. Then $\mathcal{D}_n$ is a compact and convex set in $\R^{n\times n}$ under the Frobenius norm. Define the following functional for all doubly stochastic matrix $\mb{W}$:
	\begin{align}
	\phi(\mb{W}) := \sum_{1\leq i,j\leq n}w_{i,j}(\lambda_{1,i}-\lambda_{2,j})^2,
	\end{align} 
	then $\phi$ is a linear (hence convex) functional w.r.t. to $\mb{W}$ and
	\begin{align*}
	\|\mb{M}-\mb{N}\|_F^2\geq\min\limits_{\mb{W}\in\mathcal{D}_n}\phi(\mb{W}).
	\end{align*}
	By the celebrated theorem due to Birkhoff and von Neumann which states that the convex hull of permutation matrices in $\R^{n\times n}$ is $\mathcal{D}_n$, we know the that minimum of $\phi(\mb{W})$ is attained when $\mb{W}$ is a permutation matrix. Thus
	\begin{align}
	\min\limits_{\mb{W}\in\mathcal{D}_n}\phi(\mb{W}) = \min\limits_{\sigma\in S_n}\sum_{i=1}^n(\lambda_{1,\sigma(i)}-\lambda_{2,i})^2.
	\end{align}
	Here $S_n$ is the $n$th order permutation group. Remember that $\lambda_{1,i}$ and $\lambda_{2,i}$ are in descending order, by the rearrangement inequality, 
	\begin{align}
	\sum_{i=1}^n\lambda_{1,\sigma(i)}\lambda_{2,i}\leq \sum_{i=1}^n\lambda_{1,i}\lambda_{2,i}.
	\end{align} 
	Thus
	\begin{align*}
	\|\mb{M}-\mb{N}\|_F^2\geq \min\limits_{\mb{W}\in\mathcal{D}_n}\phi(\mb{W}) = \sum_{i=1}^n(\lambda_{1,i}-\lambda_{2,i})^2.
	\end{align*}
	This proves the inequality. The equality holds if and only if $\mb{W}=\mb{I_n}$, i.e. $\mb{M}$ and $\mb{N}$ can be diagonalized simultaneously and the eigenvalues are paired in order. 
\end{proof}

\subsubsection{Section \ref{sec:ode}}
\begin{proof} [Proof of Lemma \ref{lmm:lyapnov}] We prove
	$(i)$. First, we have 
	\begin{align} \label{eq:derivativeE}
	-\mathcal{E}'(\mb{Q};\mb{A},\mb{N}) = -2\mb{AQN}.
	\end{align}
	Then, using $\mathrm{tr}(\mb{MN})=\mathrm{tr}(\mb{MN}),\ \mb{M,N}\in\R^{n\times n}$ and $\mathrm{tr}(\mb{M})=\mathrm{tr}(\mb{M}^T)$, we know that 
	\begin{align*}
	\mathrm{tr}(\mb{N\dot{Q}}^T\mb{AQ})=\mathrm{tr}(\mb{AQ}\mb{N\dot{Q}}^T)=\mathrm{tr}(\mb{\dot{Q}N}\mb{Q}^T\mb{A})=\mathrm{tr}(\mb{NQ}^T\mb{A}\mb{\dot{Q}}),
	\end{align*}
	thus according to \eqref{def:F}, we have
	\begin{align*}
	\dfrac{\dd\mathcal{E}(t)}{\dd t} &= -\mathrm{tr}(\mb{N\dot{Q}}^T\mb{AQ})-\mathrm{tr}(\mb{NQ}^T\mb{A}\mb{\dot{Q}})\\
	&= -2\mathrm{tr}(\mb{AQN}\mb{\dot{Q}}^T)\\
	&= -2\mathrm{tr}\left( \sum_{j=1}^n\sum_{k=1}^{j-1}\mb{AQNE_k}\mb{Q}^T\mb{AQE_j}\mb{Q}^T - \mb{AQNE_j}\mb{Q}^T\mb{AQE_k}\mb{Q}^T\right)\\
	&= -2\mathrm{tr}\left( \sum_{j=1}^n\sum_{k=1}^{j-1}\mb{E_jQ}^T\mb{AQNE_k}\mb{Q}^T\mb{AQ} - \mb{E_kQ}^T\mb{AQNE_j}\mb{Q}^T\mb{AQ}\right)\\
	&= -2\mathrm{tr}\left( \sum_{j=1}^n\sum_{k=1}^{j-1}\mb{E_j^2Q}^T\mb{AQNE_k^2}\mb{Q}^T\mb{AQ} - \mb{E_k^2Q}^T\mb{AQNE_j^2}\mb{Q}^T\mb{AQ}\right)\\
	&= -2\mathrm{tr}\left[ \sum_{j=1}^n\sum_{k=1}^{j-1}\left(\mb{E_jQ}^T\mb{AQE_k}\right)\mb{N}\left(\mb{E_k}\mb{Q}^T\mb{AQE_j}\right) -\left(\mb{E_kQ}^T\mb{AQE_j}\right)\mb{N}\left(\mb{E_j}\mb{Q}^T\mb{AQE_k}\right)\right]\\
	&= -2\sum_{j=1}^n\sum_{k=1}^{j-1}\mu_k(\mb{q_k}(t)\cdot \mb{Aq_j}(t))^2-\mu_j(\mb{q_k}(t)\cdot \mb{Aq_j}(t))^2\\
	&=-2\sum_{j=1}^n\sum_{k=1}^{j-1}(\mu_k-\mu_j)(\mb{q_k}(t)\cdot \mb{Aq_j}(t))^2\leq 0.
	\end{align*}
	The equality holds if and only if $\mb{q_k}\cdot \mb{Aq_j}=0$ for all $j\neq k, j,k=1,2,...,n$. Equivalently, $\mb{q_j},\ j=1,2,...,n$ are unit eigenvectors of $\mb{A}$, i.e., $\mb Q\in E$. This also proves $(ii)$
	
	Then we prove $(iii)$.
	Suppose that $\mb{Q}\in O(n)$ satisfies $\mb{F}(\mb Q)=0$, then
	\begin{align*}
	\mb{F(Q)}= \mb{Q}\left(
	\sum_{j=1}^n\sum_{k=1}^{j-1}\mb{E_j}\mb{Q}^T\mb{AQE_k}-\mb{E_k}\mb{Q}^T\mb{AQE_j}\right)=\mb{0}.
	\end{align*}
	Thus for all $j\neq k,j,k=1,2,...,n$, we have $\mb{E_j}\mb{Q}^T\mb{AQE_k}=0$, or equivalently $\mb{q_j}\cdot \mb{Aq_k}=0$. Thus $\mb{F}(\mb{Q})=0$ if and only if  $\mb{Q}\in E$.

\end{proof}

\begin{proof} [Proof of Lemma \ref{lmm:representation}]
	From \eqref{eq:rewriteODE}, we have
	\begin{align}
	\dfrac{\dd}{\dd t}(e^{-\mb{A}t}\mb{Q})+e^{-\mb{A}t}\mb{Q}\mb{T}(\mb{Q})=0,
	\end{align}
	or column-wisely 
	\begin{align} \label{eq:column}
	\dfrac{\dd}{\dd t}e^{-\mb{A}t}\mb{q_k} + 2\sum_{j=1}^{k-1}(\mb{q_j}\cdot \mb{Aq_k})e^{-\mb{A}t}\mb{q_j}+(\mb{q_k}\cdot\mb{Aq_k})e^{-\mb{A}t}\mb{q_k}=0,\ k=1,2,...,n.
	\end{align}
	
	First compute $\mb{q_1}$. Let $k=1$, we have
	\begin{align}
	\dfrac{\dd}{\dd t}e^{-\mb{A}t}\mb{q_1} +(\mb{q_1}\cdot\mb{Aq_1})e^{-\mb{A}t}\mb{q_1}=0.
	\end{align}
	Denote $g_{1,1}(t)=e^{-\int_0^t\mb{q_1(s)}\cdot\mb{Aq_1}(s)\dd s}$, then we have
	\begin{align}
	e^{-\mb{A}t}\mb{q_1}(t) = g_{1,1}(t)\mb{q_{1,0}}.
	\end{align}
	Here $\mb{q_{1,0}}$ is the first column of the initial value $\mb{Q_0}$ (see Section \ref{sec:notations}). By Lemma \ref{lmm:orthogonality}, we know that $\|\mb{q_1}(t)\|=1$, so
	\begin{align}
	g_{1,1}(t)=\dfrac{1}{\|e^{\mb{A}t}\mb{q_{1,0}}\|}.
	\end{align} 
	This solves $\mb{q_1}$ and we know that the following claim holds for $j=1$: there exist $g_{l,j}(t), 1\leq l\leq j$ such that 
	\begin{align} \label{eq:reprformula}
	e^{-\mb{A}t}\mb{q_j}(t)=\sum_{l=1}^j\mb{q_{l,0}}g_{l,j}(t).
	\end{align}
	
	We will prove that \eqref{eq:reprformula} holds for all $j=1,2,...,n$. Define
	\begin{align} \label{def:f}
	f_k(t):=e^{\int_0^t\mb{q_k}(s)\cdot\mb{Aq_k}(s)\dd s},\  k=1,2,...,n.
	\end{align}
	Suppose that \eqref{eq:reprformula} holds for $j=k-1, k\geq 2$, then by \eqref{eq:column} and \eqref{eq:reprformula}, we have
	\begin{align*}
	\dfrac{\dd}{\dd t}\left(f_k(t)e^{-\mb{A}t}\right) &= -2\sum_{j=1}^{k-1}f_k(t)(\mb{q_j}\cdot\mb{Aq_k})e^{-\mb{A}t}\mb{q_j}\\
	&= -2f_k(t)\sum_{l=1}^{k-1}\left[\sum_{j=l}^{k-1}g_{l,j}(t)(\mb{q_j}\cdot\mb{Aq_k})\right]\mb{q_{l,0}}.
	\end{align*}
	Integrating on both sides yields
	\begin{align}
	e^{-\mb{A}t}\mb{q_k}(t) = \dfrac{1}{f_k(t)}\mb{q_{k,0}}-\sum_{l=1}^{k-1}\left[\int_0^t\dfrac{2f_k(s)}{f_k(t)}\sum_{j=l}^{k-1}g_{j,l}(s)(\mb{q_j}(s)\cdot \mb{Aq_k}(s))\dd s\right]\mb{q_{l,0}}.
	\end{align}
	Thus \eqref{eq:reprformula} also holds for $j=k$ and we derive a iteration formula for $g_{l,j},\ 1\leq l\leq j$:
	\begin{gather} \label{eq:Giteration}
	g_{l,k}(t)=	
	\left\{
	\begin{split}
	&\dfrac{1}{f_k(t)},\ \ \ \ \ \ \ \ \ \ \ \ \ \ \ \ \ \ \ \ \ \ \ \ \ \ \ \ \ \ \ \ \ \ \ \ \ \ \ \ \ \ \ \ \ \ \ \ \ l=k,\\
	&-2\int_0^t\dfrac{f_k(s)}{f_k(t)}\sum_{j=l}^{k-1}g_{j,l}(s)(\mb{q_j}(s)\cdot \mb{Aq_k}(s))\dd s,\ 1\leq l\leq k-1.
	\end{split}\right.
	\end{gather}
	In fact, $g_{l,k}(t),1\leq l\leq k\leq n$ are exactly elements of $\mb{G}(t)$. Thus \eqref{eq:representation} holds, and $g_{k,k}(t)=1/f_k(t)>0$, i.e., elements of $\mb{G}(t)$ on the diagonal line are all positive.
	
	Finally, by Lemma \ref{lmm:orthogonality}, we have
	\begin{align}
	\mb{Q}(t)\mb{Q}^T(t) = e^{\mb{A}t}\mb{Q_0G}(t)\mb{G}^T(t)\mb{Q_0}^Te^{\mb{A}t}=\mb{I_n}.
	\end{align}
	Thus 
	\begin{align*} 
	\mb{G}(t)\mb{G}^T(t) = \mb{Q_0}^Te^{-2\mb{A}t}\mb{Q_0}.
	\end{align*}
	So \eqref{eq:G} holds.
\end{proof}

\begin{proof}[Proof of Lemma \ref{lmm:rank}]
	We prove by induction. Apparently \eqref{eq:rank} holds for $m=1$ by definition. Suppose that \eqref{eq:rank} holds for $m$, we prove that \eqref{eq:rank} also holds for $m+1$. Define
	\begin{align}
	I_{m,j}:=\{\sigma_k(\mb{Q}):\ \sigma_k(\mb{Q})\leq j,\ 1\leq k\leq m\},\  1\leq m,j\leq n.
	\end{align}
	Notice that
	\begin{align*}
	\mathrm{rank}\left(\mb{Q}[1,2,3,...,m+1;\ 1,2,...,j]\right)-\mathrm{rank}\left(\mb{Q}[1,2,3,...,m;\ 1,2,...,j]\right)= 0\ \mathrm{or}\ 1,
	\end{align*}
	so we discuss the following two cases:
	\begin{enumerate}
		\item If $j\geq\sigma_{m+1}(\mb{Q})$, we just need to prove that 
		\begin{align*}
		\mathrm{rank}\left(\mb{Q}[1,2,3,...,m+1;\ 1,2,...,j]\right)-\mathrm{rank}\left(\mb{Q}[1,2,3,...,m;\ 1,2,...,j]\right)=1.
		\end{align*}
		By definition of $\sigma$, we know that column vectors of $\mb{Q}[1,2,...,m;\ I_{m,j}]$ are linearly independent. Notice that by induction hypothesis, 
		\begin{align*}
		&\ \ \ \ \mathrm{rank}\left(\mb{Q}[1,2,3,...,m;\ 1,2,...,j]\right)\\
		&=m-\mathrm{card}(\{k:\ \sigma_k(\mb{Q})>j,\ 1\leq k\leq m\})\\
		&=\mathrm{card}(I_{m,j}).
		\end{align*}
		Thus column vectors of $\mb{Q}[1,2,...,m;\ I_{m,j}]$ form a basis of the column space of $\mb{Q}[1,2,...,m;\ 1,2,...,j]$. However, because $j\geq\sigma_{m+1}(\mb{Q})$, thus $\mb{Q}[1,2,...,m+1;\ \sigma_{m+1}(\mb{Q})]$ is linearly independent with column vectors of $\mb{Q}[1,2,...,m+1;\ I_{m,j}]$, thus \begin{align*}
		\mathrm{rank}\left(\mb{Q}[1,2,3,...,m+1;\ 1,2,...,j]\right)-\mathrm{rank}\left(\mb{Q}[1,2,3,...,m;\ 1,2,...,j]\right)= 1.
		\end{align*}
		
		\item If $j<\sigma_{m+1}(\mb{Q})$, then 
		\begin{gather*}
		\mathrm{card}(\{k:\ \sigma_k(\mb{Q})>j,\ 1\leq k\leq m+1\})-\mathrm{card}(\{k:\ \sigma_k(\mb{Q})>j,\ 1\leq k\leq m\})=1.
		\end{gather*}
		Consider integer $j'$ such that $1\leq j'\leq j$.  Because $j'<\sigma_{m+1}(\mb Q)$, so $\mb{Q}[1,2,...,m+1;\ j']$ can be linearly represented by column vectors of $\mb{Q}[1,2,...,m+1;\ \sigma_1(\mb{Q}),\ \sigma_2(\mb{Q}),\ ...,\ \sigma_{m}(\mb{Q})]$.
		Suppose that
		\begin{align} \label{eq:linearrepre}
		\mb{Q}[1,2,...,m+1;\ j']=\sum_{k=1}^my_k\mb{Q}[1,2,...,m+1;\ \sigma_k(\mb{Q})].
		\end{align}
		If focusing on the first $m$ rows, we have 
		\begin{align*}
		\mb{Q}[1,2,...,m;\ j']=\sum_{k=1}^my_k\mb{Q}[1,2,...,m;\ \sigma_k(\mb{Q})].
		\end{align*}
		By definition of $\sigma$, we know that for any $1\leq j'\leq j$, the column vector $\mb{Q}[1,2,...,m;\ j']$ can be linearly represented by column vectors of $\mb{Q}[1,2,...,m;\ I_{m,j}]$ (see the argument in (i)). Uniqueness of linear representations determines that
		\begin{align}
		y_k=0\ \mathrm{if}\ \sigma_k(\mb{Q})\notin I_{m,j},\ k=1,2,...,m.
		\end{align}
		Back to \eqref{eq:linearrepre}, we know that
		$\mb{Q}[1,2,...,m+1;\ j']$ can be linearly represented by $\mb{Q}[1,2,...,m+1;\ I_{m,j}]$. Thus 
		\begin{align*}
		\mathrm{rank}\left(\mb{Q}[1,2,3,...,m+1;\ 1,2,...,j]\right)=\mathrm{card}(I_{m,j})=\mathrm{rank}\left(\mb{Q}[1,2,3,...,m;\ 1,2,...,j]\right).
		\end{align*}
		So \begin{align*}
		\mathrm{rank}\left(\mb{Q}[1,2,3,...,m+1;\ 1,2,...,j]\right)+\mathrm{card}(\{k:\ \sigma_k(\mb{Q})>j,\ 1\leq k\leq m+1\})=m+1.
		\end{align*}
	\end{enumerate}
\end{proof}

\begin{proof}[Proof of Lemma \ref{lmm:z_mrepresentation}]
	If $m=1$, by definition $q_{1,j}=0$ for all $1\leq j<\sigma_1(\mb{Q})$. So \eqref{eq:q_m} holds.
	
	If $2\leq m\leq n$, by definition, $\mb{Q}[1,2,3,...,m-1;\ \mb{i_{m-1}}]\in\R^{(m-1)\times (m-1)}$ is invertible, so $\mb{c_m}$ is uniquely determined. Moreover, because $\mb{Q}[1,2,...,m;\ \sigma_m(\mb{Q})]$ is linearly independent with $\mb{Q}[1,2,...,m;\ \sigma_k(\mb{Q})], 1\leq k\leq m-1$, so there exists $z_m$ such that
	\begin{align}\label{eq:i_m}
	\mb{Q}[m;\ \mb{i_m}]=\mb{c_m}\mb{Q}[1,2,...,m-1;\ \mb{i_{m-1}}] + (0,0,...,0,z_m,0,...,0).
	\end{align}
	Taking determinant on both sides yields
	\begin{align}
	z_m = \dfrac{\det(\mb{Q}[1,2,3,...,m;\ \mb{i_m}])}{\det(\mb{Q}[1,2,3,...,m-1;\ \mb{i_{m-1}}])}.
	\end{align}
	The only thing left to prove is that 
	\begin{align} \label{eq:left}
	\mb{Q}[m;\ 1,2,3,...,\sigma_m(\mb{Q})]-\mb{c_m}\mb{Q}[1,2,...,m-1;\ 1,2,3,...,\sigma_m(\mb{Q})]=(0,0,...,0,z_m).
	\end{align}
	If $\sigma_m(\mb{Q})=1$, then $\mb{c_m}=\mb{0}$ and $z_m=q_{m,1}$, \eqref{eq:left} holds. If $\sigma_m(\mb{Q})\geq 1$, denote $K_m=\{k: \sigma_k(\mb{Q})\leq \sigma_m(\mb{Q})-1, 1\leq k\leq m\}$. By \eqref{eq:i_m}, all entries on $k$-th column where $k\in K_m$ are zero. So we just need to consider those entries are not on these columns.
	
	Otherwise, suppose that there exists $j\notin K_m, 1\leq j< \sigma_m(\mb{Q})$ such that the entry on the $j$-th column is non-zero. 
	Denote $\mb{B}=\mb{C_m}\mb{Q}[1,2...,m;\ 1,2,3,...,\sigma_m(\mb{Q})-1]$ where
	\begin{align}
	\mb{C_m}=
	\begin{pmatrix}
	1 & 0 & 0 & ... & 0\\
	0 & 1 & 0 & ... & 0\\
	0 & 0 & 1 & ... & 0\\
	\vdots & \vdots & \vdots & \ddots & \vdots\\
	-c_{m,1} & -c_{m,2} &  -c_{m,3} & ... & 1
	\end{pmatrix}
	\end{align}
	is in $\R^{m\times m}.$ Then the L.H.S. of \eqref{eq:left} is exactly the $m$-th row of $\mb{B}$. Meanwhile, $\mb{B}$ has same rank with $\mb{Q}[1,2,...,m;\ 1,2,3,...,\sigma_m(\mb{Q})-1]$ since $\mb{C_m}$ is invertible. By Lemma \ref{lmm:rank}, we know that
	\begin{align*}
	\mathrm{rank}(\mb{B}) &= \mathrm{rank}(\mb{Q}[1,2,...,m;\ 1,2,3,...,\sigma_m(\mb{Q})-1])\\
	&= m - \mathrm{card}(\{k: \sigma_k(\mb{Q})>\sigma_m(\mb{Q})-1, 1\leq k\leq m \})\\
	&= \mathrm{card}(\{k: \sigma_k(\mb{Q})\leq \sigma_m(\mb{Q})-1, 1\leq k\leq m\})\\
	&= \mathrm{card}(K_m).
	\end{align*}
	By definition of $\sigma$, we know that $\mb{B}[1,2,...,m;\ \sigma_k(\mb{Q})]$ are linearly independent if $k\in K$. Meanwhile, for all $k\in K_m$, by \eqref{eq:i_m}, $\mb{B}[m;\ \sigma_k(\mb{Q})]=\mb{0}$. Remember that we assume $b_{m,j}$ is non-zero for some $j<\sigma_m(\mb{Q})$, so $\mb{B}[1,2,3,...,m; j]$ is linearly independent with $\mb{B}[1,2,...,m;\ \sigma_k(\mb{Q})],\ k\in K_m$, hence 
	\begin{align}
	\mathrm{card}(K_m)=\mathrm{rank}(\mb{B})\geq 1+\mathrm{card}(K_m)
	\end{align}
	which is a contradiction. So \eqref{eq:left} holds.
\end{proof}
\begin{remark}
	The matrix $\mb{C_m}$ can be interpreted as row transformations on $\mb{Q}[1,2,...,m;\ 1,2,...,\sigma_m(\mb{Q})-1]$ to clear all entries on the $m$-th row, the $k$-th column, $k\in K_m$. The aim of the last part of the proof is to prove that not only these columns, but also all columns are cleared.
\end{remark}

\subsection{Linearization near stable points} \label{subsec:linearization} We will prove that the asymptotic stable points of \eqref{eq:ndODE} are $\mb Q^*$ defined in \eqref{eq:stable}. Without loss of generality, we just consider the case where $\epsilon_i=1,\ i=1,2...,n$ in \eqref{eq:ndODE}.

Consider equilibrium $\mb {Q_1}=[e_{\tau(1)},e_{\tau(2)},...,e_{\tau(n)}]$, here $\tau(i), i=1,2,...,n$ is a permutation of $1,2,...,n$. We will prove that if $\mb {Q_1}$ is table, then $\tau(i)=i$ holds for all $i=1,2,...,n$. We first consider linearization of $\mb q_1$ which reads as
\begin{align*}
	\dot{\mb {q_1}}&=\mb A(\mb{q_1}-\mb{e}_{\tau(1)})-[(\mb {e}_{\tau(1)}^T\mb A\mb{e}_{\tau(1)})\mb{I_n}+2\lambda_{\tau(1)}\mb{e}_{\tau(1)}\mb{e}_{\tau(1)}^T](\mb{q_1}-\mb{e}_{\tau(1)})\\
	&=(\mb A-\lambda_{\tau(1)}\mb{I_n}-2\lambda_{\tau(1)}\mb{E}_{\tau(1)})(\mb{q_1}-\mb{e}_{\tau(1)}).
\end{align*}

Remember that $\mb A$ is assumed to be diagonal with entries on diagonal line aligned in a descending order, so $\mb {B_1}=\mb A-\lambda_{\tau(1)}\mb{I_n}-2\lambda_{\tau(1)}\mb{E}_{\tau(1)}$ is also diagonal. Let $\mb {B_1}=\mathrm{diag}\{b_{1,1},b_{1,2},...,b_{1,n}\}$, then 
\begin{gather}
	b_j=	
	\left\{
	\begin{split}
	&\lambda_j-\lambda_{\tau(1)},\ \ \ \ \ j\neq\tau(1),\\
	&-2\lambda_{\tau(1)},\ \ \ \ \ \ j=\tau(1).
	\end{split}\right.
\end{gather}
Because this equilibrium is stable, we have $\lambda_j<\lambda_{\tau(1)}$ for all $j\neq \tau(1)$. Thus $\lambda_{\tau(1)}$ is the largest one in $\lambda_i,\ i=1,2,...,n$. Thus $\tau(1)=1.$

Then, we prove $\tau(j)=j$ for $j\geq 2$ by induction. Suppose that $\tau(i)=i$ holds for all $1\leq i\leq j-1$. Therefore, linearization of $\mb{q_j}$ reads as 
\begin{gather}\label{eq:linearization}
	\dot{\mb{q_j}}=\left(\mb A-\lambda_{\tau(j)}\mb{I_n}-2\lambda_{\tau(j)}\mb E_{\tau(j)}-2\sum_{i=1}^{j-1}\lambda_i\mb{E_i}\right)(\mb{q_j}-\mb e_{\tau(j)})-2\left(\sum_{i=1}^{j-1}\lambda_{\tau(j)}\mb E_{\tau(j),i}(\mb{q_i}-\mb{e_i})\right).
\end{gather}
Again, $\mb{B_j}=\mb A-\lambda_{\tau(j)}\mb{I_n}-2\lambda_{\tau(j)}\mb E_{\tau(j)}-2\sum_{i=1}^{j-1}\lambda_i\mb{E_i}$ is a diagonal matrix. Let $\mb{B_j}=\{b_{j,1},b_{j,2},...,b_{j,n}\}$, then 
\begin{gather}
b_{j,i}=
\left\{
	\begin{split}
	&-\lambda_{\tau(j)}-\lambda_i,\ i<j,\\
	&\lambda_i-\lambda_{\tau(j)},\ \ \ \ \ i\geq j, i\neq \tau(j),\\
	&-2\lambda_{\tau(j)},\ \ \ \ \ \ i=\tau(j).
	\end{split}\right.
\end{gather}
Again, $b_{j,i}, i=1,2,...,n$ are negative due to asymptotic stability. Thus $\lambda_i<\lambda_{\tau(j)}$ holds for all $i\geq j, i\neq\tau(j)$. Thus $\tau(j)=j$.

Therefore, by induction, we proved that $\tau(i)=1, i=1,2,...,n.$ So $\mb {Q_1}$ is asymptotic stable if and only if it is in the set defined in \eqref{eq:stable}.

\begin{remark}
	In \eqref{eq:linearization}, we see that residues $\mb{q_i}-\mb{e_i}.i=1,2,...,j-1$ also contribute to the evolution of $\mb{q_j}$. This indicates that the convergence rate of $\mb{q_j}$ depends on those of $\mb{q_i},i=1,2,...,j-1$ and can not exceed them. This is clearly observed in Theorem \ref{thm:expconv}
\end{remark}
\end{document}